\documentclass[review]{elsarticle}

\usepackage{lineno}
\usepackage{hyperref}
\usepackage{amssymb,amsmath,amsthm,mathtools}
\usepackage{changepage}
\usepackage[english]{babel}
\usepackage[T1]{fontenc}
\usepackage{microtype}
\usepackage[textsize=tiny]{todonotes}
\usepackage{tikz}
\usetikzlibrary{graphs,graphs.standard,calc,math}
\usetikzlibrary{arrows, arrows.meta}
\usetikzlibrary{automata,arrows,positioning,calc}
\usepackage{subcaption}
\usepackage{graphicx}
\usepackage{caption}
\usepackage{pgfplots}

\usepackage{enumitem}
\setlist{noitemsep,labelwidth=*,leftmargin=*,align=left}
\setlist[enumerate,1]{label=(\alph*)}
\setlist[description]{font=\normalfont,leftmargin=!}

\usepackage[ruled,noend]{algorithm2e}
\renewcommand{\SetProgSty}[1]{\renewcommand{\ProgSty}[1]{\textnormal{\csname#1\endcsname{##1}}\unskip}}
\SetProgSty{}
\SetFuncSty{textsc}
\SetFuncArgSty{}
\SetArgSty{}
\SetCommentSty{emph}
\DontPrintSemicolon

\SetKwComment{tcp}{\textup{\texttt{\#}} }{}
\SetKwComment{tcc}{''' }{ '''}
\SetKwProg{Def}{def}{:}{}

\newcommand{\algass}{\ensuremath{\leftarrow}}

\renewcommand{\O}[1]{\ensuremath{\mathcal{O}(#1)}}

\SetKw{Continue}{continue}
\SetKw{Break}{break}

\SetKwInput{Input}{Input}
\SetKwInput{Output}{Output}
\SetKw{Return}{\textbf{Return:}}
\SetKw{And}{\textbf{and}}
\SetKwComment{Comment}{/* }{ */}

\usepackage[capitalise]{cleveref}
     
\theoremstyle{definition}
\newtheorem{theorem}{Theorem}[section]

\newtheorem{corollary}[theorem]{Corollary}
\newtheorem{definition}[theorem]{Definition}
\newtheorem{remark}[theorem]{Remark}
\newtheorem{lemma}[theorem]{Lemma}
\newtheorem{problem}[theorem]{Problem}

\newtheorem{observation}[theorem]{Observation}
\newtheorem{notation}[theorem]{Notation}
\newtheorem{assumption}[theorem]{Assumption}
\theoremstyle{definition}
\newtheorem{example}[theorem]{Example}

\newcommand{\m}{m}
\newcommand{\M}{\ensuremath{\mathcal{M}}}

\newcommand{\F}{\ensuremath{\mathcal{F}}}
\newcommand{\Set}[1]{\left\{ #1 \right\}}

\pgfplotsset{compat=1.18}

\journal{Journal of ...}

\biboptions{authoryear}

\begin{document}
	
\begin{frontmatter}
		
    \title{The Parametric Matroid $\ell$-Interdiction Problem}

    \author[a]{Nils Hausbrandt}\corref{mycorrespondingauthor}
    \cortext[mycorrespondingauthor]{Corresponding author}
    \ead{nils.hausbrandt@math.rptu.de}
    \author[a]{Stefan Ruzika}
    \ead{stefan.ruzika@math.rptu.de}
      
    \address[a]{Department of Mathematics, University of Kaiserslautern-Landau, 67663 Kaiserslautern, Germany}
    
    \begin{abstract}
        In this article, we introduce the parametric matroid $\ell$-interdiction problem, where $\ell\in\mathbb{N}_{>0}$ is a fixed number of elements allowed to be interdicted.
        Each element of the matroid's ground set is assigned a weight that depends linearly on a real parameter from a given interval.
        The goal is to compute, for each possible parameter value, a set of $\ell$-most vital elements with corresponding objective value the deletion of which causes a maximum increase of the weight of a minimal basis.
        We show that such a set, which of course depends on the parameter, can only change polynomially often if the parameter varies.
        We develop several exact algorithms to solve the problem that have polynomial running times if an independence test can be performed in polynomial time.
    \end{abstract}
    
    
\end{frontmatter}

\section{Introduction}
In this article, we address the three major research areas of matroid theory (\cite{wilson1973introduction,welsh2010matroid,oxley2011matroid}), interdiction problems (\cite{smith2013modern,smith2020survey}), and parametric optimization (\cite{eisner1976mathematical,carstensen1983complexity,bazgan2022approximation}).
The present work can be seen as a continuation of our article \cite{hausbrandt2024parametric} in which these three areas were combined for the first time. We refer to this article for a detailed survey of the literature in these research fields. For any two of them, there exist several articles at the intersection of these areas, see, among others, \cite{gusfield1979bound,eppstein1995geometric,fernandez1996using,agarwal1998parametric,katoh2002parametric,eppstein2023stronger} for parametric matroid problems, \cite{frederickson1998algorithms,joret2015reducing,linhares2017improved,chestnut2017interdicting,ketkov2024class,weninger2024interdiction} for interdiction versions of arbitrary matroids, and \cite{lin1994single} for a specific variant of a parametric shortest path interdiction problem.

However, not much has been done in the combination of all three areas.
Recently, \cite{hausbrandt2024parametric} have introduced the parametric matroid one-interdiction problem, where each element of the matroid's ground set is associated with a weight that depends linearly on a real parameter.
The goal is to find, for each possible parameter value, an element that, when being removed, maximizes the weight of a minimum weight basis.

This article is intended to continue this work, with the major difference that an arbitrary but fixed number $\ell\in\mathbb{N}_{>0}$ of elements is allowed to be interdicted.
The resulting problem is called parametric matroid $\ell$-interdiction problem.
The problem aims to find, for each parameter value, a set of $\ell$-most vital elements that, when being removed, increases the weight of a minimum weight basis as much as possible.
Furthermore, the goal is to determine the piecewise linear optimal interdiction value function that maps each parameter value to the weight of an optimal $\ell$-interdicted minimum weight basis.
The complexity of the problem is measured in the number of points of slope change (changepoints) of this function since either the set of $\ell$-most vital elements or the optimal $\ell$-interdicted minimum weight basis changes.
We show that there are at most $\mathcal{O}(\m^{\ell+1} k^{\frac{1}{3}}\alpha(k\ell))$ and $\O{m^2(k+\ell)^{\ell-1}k}$ many changepoints, where $m$ is the number of elements of the matroid, $k$ is the rank of the matroid, and $\alpha$ is a functional inverse of Ackermann's function.
We develop three exact algorithms that solve the problem in polynomial time if a single independence test can be performed in time polynomial in the input length.

In \cref{sec: prelims}, we provide the preliminaries including a formal definition of the parametric matroid $\ell$-interdiction problem.
\cref{sec: structure} extends known results from non-parametric graphical matroids to arbitrary matroids with parametric weights.
We show how these results lead to a refined analysis of the number of changepoints of the optimal interdiction value function.
The resulting tighter bounds on the number of changepoints imply faster running times of our algorithms, which are developed in \cref{sec: algorithms}.
    
\section{Preliminaries}\label{sec: prelims}
We introduce definitions and notations and recall some basic results for the parametric matroid one-interdiction problem.
This exposition also extends the results of \cite{hausbrandt2024parametric} to the case that not only one but an arbitrary number $\ell\in\mathbb{N}_{>0}$ of elements can be interdicted.
For two sets $A$ and $A'$ with $A \cap A'=\emptyset$, we denote the disjoint union of $A$ and $A'$ by $A\dot{\cup}A'$.
For a set $A$ and a singleton $\Set{a}$, we write $A-a$ or $A+a$ for $A\setminus\Set{a}$ or $A\cup\Set{a}$, respectively.

\textbf{Matroids.}
For a finite set $E$, a tuple $\M=(E,\F)$ with $\emptyset\neq\F\subseteq 2^E$ is called \emph{matroid} if the following properties hold:
\begin{enumerate}
    \item The empty set $\emptyset$ is contained in $\F$.
    \item If $A\in\F$ and $B\subseteq A$, then also $B\in\F$.
    \item If $A,B \in\F$ and $\vert B\vert < \vert A\vert$, then there exists an element $a\in A\setminus B$ such that $B + a\in\F$.
\end{enumerate}
The elements of $\F$ are called \emph{independent} sets of $\M$, while all other subsets of $E$ are called \emph{dependent}.
A \emph{basis} of $\M$ is an inclusion-wise maximum independent set of $\M$.
All bases have the same cardinality which is called the \emph{rank} $rk(\M)$ of $\M$.
A \emph{circuit} is a minimal dependent set.
We denote the cardinality of $E$ by $m$ and the rank of $\M$ by $k\coloneqq rk(\M)$.

For a subset $E'\subseteq E$, we denote the matroid $(E',\F')$ with $\F'\coloneqq\left\{F\in\F\colon\, F\subseteq E' \right\}$ by $\M|E'$.
For $F\subseteq E$, we write $\M_F$ for $\M|(E\setminus F)$ and, if $F=\Set{e}$ is a singleton, we write $\M_e$ for $\M|(E-e)$.

\textbf{Parametric matroids.}
In our setting, each element $e\in E$ is associated with a parametric weight $w(e,\lambda)= a_e + \lambda b_e$, where $a_e,b_e \in \mathbb{Q}$.
The parameter $\lambda$ is taken from a real interval $I\subseteq\mathbb{R}$, called the \emph{parameter interval}.
The weight of a basis $B$ is defined as $w(B,\lambda)\coloneqq\sum_{e\in B} w(e,\lambda)$.

In the parametric matroid problem, the goal is to compute a minimum weight basis $B_\lambda^*$ for each parameter value $\lambda \in I$ .
The function $w:I\to \mathbb{R}, \lambda\mapsto w(B_\lambda^*,\lambda)$ is called \emph{optimal value function}.
It is well known that $w$ is piecewise linear and concave, cf.\ \cite{gusfield1980sensitivity}.
A \emph{breakpoint} is a point $\lambda\in I$ at which the slope of $w$ changes.
For the parametric matroid problem, there is a tight bound of $\Theta(\m k^{\frac{1}{3}})$ on the number of breakpoints, cf.\ \cite{dey1998improved,eppstein1995geometric}.
A breakpoint can only occur at an \emph{equality point}, which is a point $\lambda\in I$ at which two weight functions $w(e,\lambda)$ and $w(f,\lambda)$ become equal.
Clearly, there are at most $\binom{m}{2} \in \O{m^2}$ many equality points.
In the following sections, we consider specific equality points.
To this end, let $\lambda(e,f)$ be the equality point where $w(e,\lambda) = w(f,\lambda)$.
If $w(e,\lambda)$ and $w(f,\lambda)$ never become equal, we set $\lambda(e,f)$ to $-\infty$.
We also write $\lambda(e\to f)$ for the equality point $\lambda(e,f)$, for which $w(e,\lambda) < w(f,\lambda)$ for $\lambda < \lambda(e,f)$ and, consequently, $w(e,\lambda) > w(f,\lambda)$ for $\lambda > \lambda(e,f)$.
We do not need to consider the case that $w(e,\lambda) = w(f,\lambda)$ for all $\lambda\in I$, since we can exclude it later in \cref{ass: l-interdiction}.

There is a simple algorithm for the parametric matroid problem. First, all equality points are computed and sorted in ascending order.
Before the first equality point, i.~e.\ for a value $\lambda$ that is smaller than the smallest equality point, a minimum weight basis $B_\lambda^*$ can be computed using the well-known greedy algorithm.
Then, at each equality point $\lambda(e\to f)$, an independence test of $B_\lambda^*-e+f$ is performed to obtain an $\mathcal{O}(\m^2(f(\m)+\log m))$ algorithm.
Here, $f(\m)$ is the time needed to perform a single independence test.

\textbf{Interdicting parametric matroids.}
Let $\ell\in \mathbb{N}_{>0}$ be the number of elements allowed to be interdicted.
We summarize our notation.
\begin{notation}
    Throughout this article, $\M=(E,\F)$ is a matroid with parametric weights $w(e,\lambda)=a_e+\lambda b_e$ for $\lambda\in I\subseteq \mathbb{R}$, where
    \begin{enumerate}
        \item[] $m$ is the cardinality of $E$,
        \item[] $k$ is the rank of $\M$, and
        \item[] $\ell$ is the number of elements allowed to be interdicted.
    \end{enumerate}
\end{notation}
\begin{definition}[Set of $\ell$-most vital elements]\label{def: set of l mves}
    Let $\lambda\in I$.
    For a subset $F\subseteq E$, we denote a minimum weight basis on $\M_F$ at $\lambda$ by $B_\lambda^F$.
    If $F=\Set{e}$, we write $B_\lambda^e$.
    If $\M_F$ does not have a basis of rank $k$, we set $w(B_\lambda^F,\lambda)=\infty$ for all $\lambda\in I$.
    A subset $F^*\subseteq E$ with $\vert F^*\vert=\ell$ is called a \emph{set of $\ell$-most vital elements} at $\lambda$ if $w(B_\lambda^{F^*},\lambda) \geq w(B_\lambda^F,\lambda)$ for all $F\subseteq E$ with $\vert F\vert=\ell$.    
\end{definition}
Although we generally consider the case $\ell>1$ in this article, it turns out to be helpful to interdict single elements optimally.
\begin{definition}[Most vital element]
    Let $\lambda\in I$.
    An element $e^*\in E$ is called \emph{most vital element} at $\lambda$ if $w(B_\lambda^{e^*},\lambda) \geq w(B_\lambda^e,\lambda)$ for all $e\in E$.
\end{definition}
\begin{definition}[Optimal interdiction value function] \label{def: opt l-interdiction value function}
    For $F\subseteq E$ with $\vert F\vert=\ell$, we define the function $y_F$ by $y_F\colon I\to\mathbb{R}$, $\lambda\mapsto w(B_\lambda^F,\lambda)$ mapping the parameter $\lambda$ to the weight of a minimum weight basis of $\M_F$ at $\lambda$.
    For $\lambda\in I$, we define $y(\lambda)\coloneqq\max\Set{y_F(\lambda)\colon\, F\subseteq E,\;\vert F\vert=\ell}$ as the weight of an optimal $\ell$-interdicted matroid at $\lambda$.
    The \emph{optimal interdiction value function} $y$ is then defined via 
    $y\colon I\to\mathbb{R}$, $\lambda\mapsto y(\lambda)$.
\end{definition}
We are now ready to formulate the parametric matroid $\ell$-interdiction problem.
\begin{problem}[Parametric matroid $\ell$-interdiction problem]\label{prob: l-matroid}
    Given a matroid~$\M$ with parametric weights $w(e,\lambda)$, a parameter interval~$I$, and a number $\ell\in\mathbb{N}_{>0}$, the goal is to determine, for each $\lambda\in I$, a set of $\ell$-most vital elements $F^*$ and the corresponding objective function value $y(\lambda) = y_{F^*}(\lambda)$.
\end{problem}
\cref{prob: l-matroid} is already $\mathsf{NP}$-hard for a fixed parameter value $\lambda \in I$.
This follows from the $\mathsf{NP}$-hardness of the special case of the $\ell$-most vital edges problem with respect to graphical matroids, cf.\ \cite{frederickson1999increasing}.
In this article, we focus on the computation of an exact solution to the problem, i.~e.\ a set of $\ell$-most vital elements with corresponding objective value for each parameter value $\lambda\in I$.
We therefore assume that $\ell$ is constant and not part of the input.
Nevertheless, we investigate the influence of the parameter $\ell$ on the running time of our algorithms.
\begin{assumption}\label{ass: l-interdiction} We make the following assumptions.
    \begin{enumerate}
        \item The number $\ell$ of elements allowed to be interdicted is constant.
        \item \label{ass: b} There exists a basis $B_\lambda^F$ of cardinality $k$ for every $\lambda\in I$ and $F\subseteq E$ with $\vert F\vert=\ell$.
        \item There exists unique optimal bases $B_\lambda^*$ and $B_\lambda^F$ for every $\lambda\in I$ and $F\subseteq E$ with $\vert F\vert=\ell$.\label{ass: c}
        \item No two pairs of weights $w(e,\lambda)$ become equal simultaneously.\label{ass: d}
    \end{enumerate}
\end{assumption}
Assumptions \ref{ass: b}-\ref{ass: d} are without loss of generality.
Assumption \ref{ass: b} excludes the trivial case that, after removing a set $F$ of $\ell$ elements, the matroid $\M_F$ has no basis with rank $k$.
If there exists such a set $F$, then it is a set of $\ell$-most vital elements with objective value $y_F(\lambda)$ equal to infinity for each parameter value $\lambda\in I$.
We continue with Assumption \ref{ass: c}.
For any point $\lambda\in I$ which is not an equality point, the optimal bases $B_\lambda^*$ and $B_\lambda^F$ already have a unique weight.
At an equality point, ties can be solved by an arbitrary but fixed ordering of the elements $e\in E$.
We consider Assumption \ref{ass: d}.
If three or more weights become equal at an equality point $\lambda$, a sufficiently small $\varepsilon >0$ can be added to all but two weights.
This minimal change also minimally influences the optimal bases $B_\lambda^*$ and $B_\lambda^F$ for $F\subseteq E$.
This means that all functions $y_F$ and, therefore, also their point-wise maximum $y$ are not significantly changed.
We also refer the reader to \cite{fernandez1996using}, where this assumption was made in the context of parametric minimum spanning trees.
Without loss of generality, we can also exclude the case that two weight functions $w(e,\lambda)$ and $w(f,\lambda)$ are equal for all $\lambda\in I$.
Otherwise, all weight functions are parallel due to Assumption \ref{ass: d} and \cref{prob: l-matroid} is reduced to the non-parametric variant with a fixed $\lambda$.

\cref{ass: l-interdiction} implies the following observation and the subsequent definition in analogy to the case of $\ell=1$.
\begin{observation}\label{obs: y is upper env continuous}
    The optimal interdiction value function $y$, which is the upper envelope of the functions $y_F$, is piecewise linear and continuous. 
\end{observation}
\begin{definition}[Changepoints]
    \label{def: l-change-break-interdiction-points}
    The points of slope change of $y$ are called \emph{changepoints} and are partitioned into \emph{breakpoints} and \emph{interdiction points}. 
    A breakpoint $\lambda$ of $y$ occurs if a function $y_{F^*}$, where $F^*$ is a set of $\ell$-most vital elements before and after $\lambda$, has a breakpoint. 
    A point $\lambda$ is an interdiction point of $y$ if the set of $\ell$-most vital elements changes at $\lambda$. This corresponds to the case that a function $y_{F^*}$ intersects a function $y_{G^*}$, where $F^*$ and $G^*$ are different sets of $\ell$-most vital elements before and after $\lambda$, respectively.
\end{definition}
Note that $y_F$ is the optimal value function of the matroid $\M_F$ and, according to Assumption \ref{ass: b}, $y_F$ is continuous and piecewise linear.
Consider a set $F^*$ of $\ell$-most vital elements before and after a point $\lambda$ such that the function $y_{F^*}$ has a breakpoint at $\lambda$.
This point of slope change is then transferred to the function $y$ and we therefore also call these points breakpoints.
The situation is illustrated in \cref{fig: changepoints}.

\begin{figure}
    \centering
    \resizebox{9cm}{6cm}{
    \begin{tikzpicture}[
        declare function={
		F1(\x)= (\x < 4) * (2+0.5*\x) + (\x >= 4) * (8-1*\x);
		F2(\x)= (\x < 2) * (1.8+1.1*\x) + (\x >= 2) * (4.4-0.2*\x);
		F3(\x)= (\x < 1) * (2+1.5*\x) + and(\x >= 1, \x < 3) * (3.25+0.25*\x) + (\x >= 3) * (5.5-0.5*\x);
		Y(\x)= (\x < 0) * (2+0.5*\x)
        + and(\x >= 0,\x < 1) * (2+1.5*\x)
        + and(\x >= 1,\x < 1.705) * (3.25+0.25*\x)
        + and(\x >= 1.705,\x < 2) * (1.8+1.1*\x)
        + and(\x >= 2,\x < 2.55) * (4.4-0.2*\x)
        + and(\x >= 2.55,\x < 3) * (3.25+0.25*\x)
        + and(\x >= 3,\x < 3.55) * (5.5-0.5*\x)
        + and(\x >= 3.55,\x < 4) * (2+0.5*\x)
        + and(\x >= 4,\x < 4.5) * (8-1*\x)
        + (\x >= 4.5) * (4.4-0.2*\x);
	    },
    	]
		\begin{axis}[
			axis x line=middle, axis y line=middle,
			ymin=-0.5, ymax=5, ytick={},
			xmin=-1, xmax=9, xtick={},
            xtick={1,4.5},
            xticklabels={$\lambda'$,$\lambda''$},
            x label style={at={(axis description cs:0.98,-0.01)}},
            xlabel=$\lambda$,
            ytick={-2},
            ylabel=$y(\lambda)$,
            y label style={at={(axis description cs:-0.04,0.98)}},
			domain=-1:7,
			]
            \addplot [gray!50,line width=2.7] {Y(x)} node [pos=1,right] {};
			\addplot [dashed,very thick,] {F1(x)} node [pos=1,below right] {$y_{F_1}$};
            \addplot [loosely dashdotted,very thick,] {F2(x)} node [pos=1,below right] {$y_{F_2}$};
			\addplot [dotted,very thick] {F3(x)} node [pos=1,right] {$y_{F_3}$};
            \addplot[mark=none, dashed] coordinates {(1, 0) (1, 3.5)};
            \addplot[mark=none, dashed] coordinates {(4.5, 0) (4.5, 3.5)};
            \node[gray] at  (7.4,3.2) (B){$y$};
		\end{axis}
    \end{tikzpicture}
    }
    \caption{The figure shows three possible optimal value functions $y_{F_1}, y_{F_2}$, and $y_{F_3}$ for different subsets $F_1,F_2$, and $F_3$.
    The objective function $y$ is given by the upper envelope of the functions $y_{F_i}$ and its points of slope change are the changepoints.
    These are subdivided into breakpoints and interdiction points.
    The point $\lambda'$ is a breakpoint of $y_{F_3}$ and, since $y_{F_3}$  forms the upper envelope at this point, $\lambda'$ is also a breakpoint of $y$.
    The point $\lambda''$ is an interdiction point, as the functions $y_{F_1}$ and $y_{F_2}$ intersect in this point and before $\lambda''$, the upper envelope equals $y_{F_1}$, and after $\lambda''$, the upper envelope equals $y_{F_2}$.
    Consequently, the set of $\ell$-most vital elements changes accordingly from $F_1$ to $F_2$ at $\lambda''$.
    }
    \label{fig: changepoints}
\end{figure}
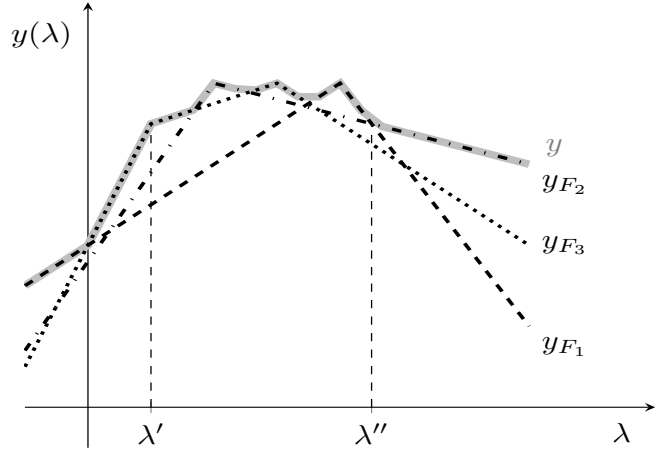

Note that the changepoints are exactly the points at which the solution of \cref{prob: l-matroid} changes.
Either the optimal interdiction strategy $F^*$ changes or the optimal $\ell$-interdicted minimum weight basis $B_\lambda^{F^*}$ changes.
In particular, when solving \cref{prob: l-matroid}, it is sufficient to compute one solution for a fixed value $\lambda$ between any two consecutive changepoints of $y$.
When computing bounds on the number of changepoints, we often handle the case $I=\mathbb{R}$.
This is without loss of generality, too, since there cannot be more changepoints on a subinterval of $\mathbb{R}$.

The concept of replacement elements is essential to measure the effect of interdicting one or more elements.
\begin{definition}[Replacement element]
    Let $\lambda \in I$.
    For $e\in B_\lambda^*$, we define the set of all \emph{replacement candidates} $R_\lambda(e)$ of $e$ at~$\lambda$ in the minimum weight basis as $R_\lambda(e)\coloneqq\Set{r\in E\setminus B_\lambda^* \;\colon\; B_\lambda^*-e+r \in \F}$.
    The \emph{replacement element} of $e$ at $\lambda$ is defined as $r_\lambda(e) = \arg\min\Set{w(r,\lambda)\colon\, r\in R_\lambda(e)}$.
\end{definition}
Later, we also compute replacement elements of elements of an interdicted basis $B_\lambda^F$ for $F\subseteq E$ with $\vert F\vert\leq\ell$.
So, if the context is not clear, we specify the basis for which the replacement elements are determined.
Note that the most vital element $e^*$ at $\lambda$ is given by the maximum difference between the weights of the elements of $B_\lambda^*$ and their replacement elements, that is $e^*=\arg\max\Set{w(r_\lambda(e),\lambda)-w(e,\lambda)\colon e\in B_\lambda^*}$.
The following lemma ensures that a replacement element $r_\lambda(e)$ actually replaces $e$ at $\lambda$ if $e$ is removed from $B_\lambda^*$.
\begin{lemma}[\cite{hausbrandt2024parametric}]\label{lem: replacement element}
    Let $\lambda \in I$ and $e\in B_\lambda^*$. Let $r_\lambda(e)$ be the replacement element of $e$ at $\lambda$ with respect to $B_\lambda^*$.
    Then, $B_\lambda^e = B_\lambda^* -e + r_\lambda(e)$.
\end{lemma}
The candidates for a most vital element at a point $\lambda$ can be restricted to elements of the minimum weight basis $B_\lambda^*$.
\begin{lemma}[\cite{hausbrandt2024parametric}]\label{lem: for mve only elements of B*}
    For any $\lambda\in I$, a most vital element is an element of $B_\lambda^*$.
\end{lemma}
To compare the running times of our algorithms, we define the time for a single independence test and the time to compute a replacement element of a basis element.
For the latter, we also give an amortised version, as we later iteratively compute the replacement element for all elements of a basis.
\begin{definition}
    We define the following runtimes for operations on matroids:
    \begin{description}
        \item[$f(\m)$] the time needed to perform a single independence test.
        \item[$h(\m)$] the time needed to compute a replacement element $r_\lambda(e)$ for a given $\lambda\in I$ and $e\in B_\lambda^*$.
        \item[$H(m)$] the amortised cost for computing all $k$ replacement elements for a given $\lambda\in I$ and elements $e\in B_\lambda^*$.
    \end{description}
\end{definition}
The running time without amortised costs can easily be obtained since $H(m)\in \O{kh(m)}$.

There is an important connection between the theory of parametric matroids and the theory of matroid interdiction.
If an element $e$ leaves the optimal basis $B_\lambda^*$ at a breakpoint $\lambda(e\to f)$, then $e$ is swapped with its replacement element, i.~e.\ it holds that $r_\lambda(e)=f$ for $\lambda < \lambda(e\to f)$.
\begin{lemma}\label{lem: r(e) = f before swap}
    Let $\lambda_{i-1} <\lambda_i=\lambda(e\to f)<\lambda_{i+1}$ be three consecutive equality points.
    Then, $\lambda_i$ is a breakpoint of $w$ if and only if $r_\lambda(e) = f$ for $\lambda \in (\lambda_{i-1} ,\lambda_i]$ and $r_\lambda(f) = e$ for $\lambda \in (\lambda_i ,\lambda_{i+1}]$.
\end{lemma}
\begin{proof}
    If $\lambda_i$ is a breakpoint of $w$, the claim follows from Lemma 3.6 in \cite{hausbrandt2024parametric}.
    For $\lambda \in (\lambda_{i-1} ,\lambda_i]$, if $r_\lambda(e) = f$, we can conclude that $e\in B_\lambda^*$ and $f\notin B_\lambda^*$.
    For $\lambda \in (\lambda_i ,\lambda_{i+1}]$ with $r_\lambda(f) = e$, it follows that $f\in B_\lambda^*$ and $e\notin B_\lambda^*$.
    This is only possible if $\lambda_i$ is a breakpoint of $w$.
\end{proof}
For a fixed $\lambda\in I$ and a subset $F\subseteq E$, the objective function value $y_F(\lambda)=w(B_\lambda^F,\lambda)$ can of course be computed with the greedy algorithm on $\M_F$.
However, this is also possible by successively deleting and replacing the elements of $F$ from $B_\lambda^*$ in an arbitrary order, which proves helpful in the next section.
\begin{observation}\label{obs: deleting F in arbitrary order}
    Let $F\subseteq E$ and $\lambda\in I$.
    We obtain the optimal basis $B_\lambda^F$ by iteratively deleting and replacing the elements of $F$ in an arbitrary order from $B_\lambda^*$.
\end{observation}
\begin{proof}
    We show the claim by induction on $\ell=\vert F\vert$.
    We can without loss of generality assume that $F\subseteq B_\lambda^*$.
    Otherwise, we can initially delete the elements from $F\setminus B_\lambda^*$, as they are not contained in the basis $B_\lambda^F$.
    For $\ell=1$, the claim follows from \cref{lem: replacement element}.
    Consider $\ell>1$.
    We remove some element $g\in F$ from $B_\lambda^*$, replace it by $r_\lambda(g)$ and obtain the optimal basis $B_\lambda^g=B_\lambda^*-g+r_\lambda(g)$ of $\M'\coloneqq \M_g$.
    By induction, iteratively deleting and replacing the elements of $F-g$ from $B_\lambda^g$ yields an optimal basis on $\M'_{F-g}=\M_F$.
\end{proof}

\section{Structural results}\label{sec: structure}
Obviously, the number of changepoints corresponds to the number of optimal solutions and, thus, determines the running time of any exact algorithm.
In this section, we derive several properties of an optimal interdiction strategy $F^*$ that bound the number of changepoints of the optimal interdiction value function $y$.
These properties imply three different algorithms in \cref{sec: algorithms} the running times of which are determined by the different bounds on the number of changepoints obtained in this section.
We extend two existing concepts for non-parametric graphical matroids to arbitrary matroids with parametric weights.
A first bound follows directly from \cref{obs: y is upper env continuous} and the theory of Davenport–Schinzel Sequences, cf.\ \cite{agarwal1995davenport}.
\begin{theorem}\label{theo: l int upper bound 1}
    The optimal interdiction value function $y$ of the parametric matroid $\ell$-interdiction problem has at most $\mathcal{O}(\m^{\ell+1} k^{\frac{1}{3}}\alpha(m))$ many changepoints.
\end{theorem}
\begin{proof}
    It follows from \cite{dey1998improved} that each of the $\binom{\m}{\ell}$ many piecewise linear and continuous functions $y_F$ has at most $\mathcal{O}(\m k^{\frac{1}{3}})$ many breakpoints such that we obtain a total of $\mathcal{O}(\m^{\ell+1} k^{\frac{1}{3}})$ many linear pieces in the graphs of these functions.
    Corollary~2.18 from \cite{agarwal1995davenport} implies that their upper envelope $y$ has at most $\mathcal{O}(\m^{\ell+1} k^{\frac{1}{3}}\alpha(\m^\ell))$ many changepoints.
    Finally, we show that $\alpha(m^\ell)\in\O{\alpha(m)}$ using Ackermann's function, cf.~\cite{cormen2022introduction} which is defined as
    \begin{displaymath}
        A_p(j)=
        \begin{cases}
            j+1,  & p=0 \\
            A_{p-1}^{(j+1)}(j), & p\geq 1.
        \end{cases}
    \end{displaymath}
    Let $p'=\alpha(m)=\min\Set{p\colon A_p(1)\geq m}$.
    We show that $A_{p'+1}(1)\geq m^\ell$, since it then follows that $\alpha(m^\ell)=\min\Set{p\colon A_p(1)\geq m^\ell} \leq p'+1\in\O{\alpha(m)}$.
    
    It holds that $A_{p'+1}(1)=A_{p'}^{(2)}(1)>(A_{p'-1}^{(2)}(1))^\ell=(A_{p'}(1))^\ell\geq m^\ell$, where the inequality follows from
    \begin{equation*}
    	\begin{split}
    		& A_p^{(2)}(1) =A_p(A_p(1))=A_p(A_{p-1}^{(2)}(1))=A_p(A_{p-1}(A_{p-1}(1))) \\
    		& > (A_{p-1}(A_{p-1}(1)))^\ell =  (A_{p-1}^{(2)}(1))^\ell.
    	\end{split}
    \end{equation*}
\end{proof}
Next, we do not want to consider all $\binom{\m}{\ell}$ possible subsets $F\subseteq E$ of $\ell$ elements.
Instead, we restrict the candidates to so-called non-dominated subsets.
\begin{definition}
    Let $J\subseteq I$ and $F,F'\subseteq E$. The set $F$ \emph{dominates} the set $F'$ on $J$ if $y_F(\lambda) \geq y_{F'}(\lambda)$ for all $\lambda\in J$.
    A set that is not dominated on $J$ is called \emph{non-dominated} on $J$.
\end{definition}
For the remainder of this section, let $J=(\lambda_i,\lambda_{i+1})\subseteq I$, where $\lambda_i$ and $\lambda_{i+1}$ are two consecutive equality points.
Note that the optimal bases $B_\lambda^*$ and $B_\lambda^F$ as well as all replacement candidates and elements remain unchanged on $J$.
This means that if $B=B_\lambda^*$, $B'=B_\lambda^F$, $R=R_\lambda(e)$ or $r=r_\lambda(e)$ for one $\lambda\in J$, then $B=B_\lambda^*$, $B'=B_\lambda^F$, $R=R_\lambda(e)$, and $r=r_\lambda(e)$ for all $\lambda\in J$.
\begin{lemma}\label{lem: at least one element from B^*}
    Let $\lambda \in J$ and $F\subseteq E$. If $F\cap B_\lambda^*= \emptyset$, then $F$ is dominated on $J$. 
\end{lemma}
\begin{proof}
    If $F\cap B_\lambda^* = \emptyset$, it holds that $B^F = B_\lambda^*$ and $w(B^F,\lambda)= w(B_\lambda^*,\lambda)$ for all $\lambda \in J$.
    For every $F'\subseteq E$ with $F'\cap B_\lambda^* \neq \emptyset$, we get $B^{F'}\neq B_\lambda^*$ and $y_{F'}(\lambda)= w(B^{F'},\lambda) \geq w(B_\lambda^*,\lambda)=w(B^F,\lambda)=y_F(\lambda)$.
\end{proof}
\cref{lem: at least one element from B^*} shows that all subsets containing $\ell$ elements of the $m-k$ many non-basis elements $e\in E\setminus B_\lambda^*$ are redundant.
Hence, on $J$ there are at most $\binom{m}{\ell} -\binom{m-k}{\ell} $ many non-dominated subsets.
We further tighten the upper bound by generalizing the concept of sparse, weighted $\ell$-connected certificates of \cite{liang1997finding} from graphical matroids with fixed edge weights to arbitrary matroids with parametric weights.
To do this, we delete the optimal basis $B_\lambda^*$, then compute the next best basis, delete this again, and continue this procedure until we obtain $\ell$ bases.
\begin{definition}\label{def: bases Bi of Ui}
    For $\lambda\in J$, let $B_\lambda^0 \coloneqq B_\lambda^*$ and for $i>0$, let $B_\lambda^i$ be an optimal basis of $\M_{U_\lambda^{i-1}}$, where $U_\lambda^{i-1}=\bigcup_{j=0}^{i-1} B_\lambda^j$.
\end{definition}
Note that each basis $B_\lambda^{i}$ and, hence, each set $U_\lambda^{i}$ remains unchanged on $J$ as well.
When computing a set $F^*$ for some $\lambda\in J$, it suffices to focus on the set $U_\lambda^{\ell-1}$.
\begin{lemma}\label{lem: F* subset U l+1}
    Let $\lambda\in J$. If $e\in E\setminus F$ is not an element of $U_\lambda^\ell$, then $e$ is not an element of an optimal basis $B_\lambda^F$ of $\M_F$ on $J$ for any $F\subseteq E$ with $\vert F\vert\leq\ell$.
\end{lemma}
\begin{proof}
    Let $\lambda\in J$ and $F\subseteq E$ with $\vert F\vert \leq\ell$. Let $e \in E\setminus F$ with $e \notin U_\lambda^\ell$ and suppose that $e$ is in the optimal basis $B_\lambda^F$ of $\M_F$ on $J$. 
    Let $i \in \Set{0,\dotsc,\ell-1}$ be arbitrary.
    If we add $e$ to the basis $B_\lambda^i$, we get a circuit $C$ and $e$ is the heaviest element of $C$.
    If we remove $e$ from the basis $B_\lambda^F$, there is an $f\in B_\lambda^i\setminus B_\lambda^F$ such that $B_\lambda^F - e + f$ is again a basis.
    We get a contradiction if $f \in C-e$.
    Let $C-e =\Set{f_1,\dotsc,f_p} \subseteq E$.
    
    Suppose that $B_\lambda^F - e + f_j$ contains a circuit $C_j$ for all $j=1,\dotsc,p$ with $f_j\notin B_\lambda^F$. 
    Then, for each such $j$, there is an element $g_j \in C_j - f_j$ such that $B_\lambda^F - g_j + f_j$ is a basis.
    Furthermore, we can choose $g_j$ such that $g_j \notin C-e$.
    Otherwise, $f_j$ and all $g_j\in C_j-f_j$ are in $C-e$ and, thus, $C_j \subseteq C-e$, which is a contradiction, as $C-e$ is independent.

    We can therefore successively swap all elements $f_j \notin B_\lambda^F$ into the basis $B_\lambda^F$ by replacing them with $g_j$.
    Note that all elements $g_j$ and $f_j$ are pairwise different as $f_j \in C-e$ and $g_j \notin C-e$.
    Hence, after a maximum of $p$ swaps, the whole set $C-e$ is swapped into the basis $B_\lambda^F$ which is a contradiction, since $e\in B_\lambda^F$.
\end{proof}
\begin{remark}\label{rem: at most kl choose l subsets}
    \cref{lem: F* subset U l+1} says that all sets $F\subseteq E$ with $F \nsubseteq U_\lambda^{\ell-1}$ are dominated on $J$.
    Thus, there are at most $\binom{k\ell}{\ell}$ many non-dominated sets $F\subseteq E$ with $\vert F \vert =\ell$ on $J$ as each of the $\ell$ bases $B_\lambda^i$ remains unchanged on $J$ and contains $k$ elements.
\end{remark}
\begin{lemma}\label{lem: l-1 elements chosen}
    Let $\lambda \in J$ and $F'\subseteq E$ with $\vert F'\vert =\ell-1$. At $\lambda$, the best interdiction strategy $F\subseteq E$ with $\vert F \vert =\ell$ and $F'\subseteq F$ is given by $F'+e^*$ where $e^*$ is the most vital element of $\M_{F'}$ at $\lambda$.
\end{lemma}
\begin{proof}
    Let $\lambda \in J$ and $e^*$ be the most vital element of $\M_{F'}$ at $\lambda$.
    Let $F\coloneqq F'+e^*$.
    Suppose there exists a set $\hat{F}\subseteq E$ with $\vert\hat{F}\vert=\ell$ and $F'\subseteq \hat{F}$ such that $\hat{F}\setminus F'=\Set{e}\neq\Set{e^*}$ and $y_{\hat{F}}(\lambda) > y_F(\lambda)$.
    By \cref{lem: for mve only elements of B*}, we know that $e^*$ is an element of the optimal basis $B_\lambda^{F'}$ of $\M_{F'}$.
    Without loss of generality, we can assume that $e\in B_\lambda^{F'}$.
    By \cref{obs: deleting F in arbitrary order}, we obtain $B_\lambda^F=B_\lambda^{F'+ e^*} = B_\lambda^{F'}-e^*+r_\lambda(e^*)$ and $B_\lambda^{\hat{F}}=B_\lambda^{F'+ e} = B_\lambda^{F'}-e+r_\lambda(e)$, where the replacement elements $r_\lambda(e^*)$ and $r_\lambda(e)$ are with respect to the basis $B_\lambda^{F'}$.
    It holds that 
\begin{equation*}
    \begin{split}
        & w(B_\lambda^{F'},\lambda)-w(e,\lambda)+w(r_\lambda(e),\lambda) = w(B_\lambda^{\hat{F}},\lambda) \\
    & > w(B_\lambda^F,\lambda) = w(B_\lambda^{F'},\lambda)-w(e^*,\lambda)+w(r_\lambda(e^*),\lambda)
    \end{split}
\end{equation*} and, therefore, $-w(e,\lambda)+w(r_\lambda(e) > -w(e^*,\lambda)+w(r_\lambda(e^*),\lambda)$, which is a contradiction to the fact that $e^*$ is the most vital element of $\M_{F'}$ at $\lambda$.
\end{proof}
\cref{lem: F* subset U l+1} and \cref{lem: l-1 elements chosen} imply the following tighter bound on the number of changepoints of $y$.
\begin{corollary}\label{cor: bound using U}
    The optimal interdiction value function $y$ of the parametric matroid $\ell$-interdiction problem has at most $\binom{m}{2}\binom{k(\ell-1)}{\ell-1}k \in \O{m^2k^\ell\ell^{\ell-1}}$ many changepoints.
\end{corollary}
\begin{proof}
    Consider a subinterval $J\subseteq I$ between two consecutive equality points.
    By \cref{rem: at most kl choose l subsets} and \cref{lem: l-1 elements chosen}, it suffices to consider $\binom{k(\ell-1)}{\ell-1}$ many subsets $F'$ of cardinality $\ell-1$ on $J$.
    Furthermore, by \cref{lem: at least one element from B^*} there are at most $k$ most vital elements of $\M_{F'}$ on $J$.
    This yields at most $k\binom{k(\ell-1)}{\ell-1}$ many linear functions to determine $y$ on $J$.
    Consequently, we obtain at most $k\binom{k(\ell-1)}{\ell-1}-1$ many interdiction points on $J$.
    Together with a potential breakpoint at the boundary of $J$, we obtain at most $\binom{m}{2}\binom{k(\ell-1)}{\ell-1}k$ many changepoints on the whole interval $I$.
\end{proof}
In \cref{lem: l-1 elements chosen}, we see that the choice of $\ell-1$ elements determines the missing $\ell$-th element.
We show in the next lemma that the choice of already one (or more) elements can restrict the candidates for the remaining elements to be interdicted.
The result can be seen as a generalization of the method of \cite{bazgan2011efficient} for identifying the $\ell$-most vital edges of a minimum spanning tree.
\begin{lemma}\label{lem: F* is union of Fi}
    If $F^*$ is a set of $\ell$-most vital elements for some $\lambda \in J$, there exists a partition $F^*=\dot{\bigcup}_{i=1}^t F_i$ for some $1\leq t\leq\ell$ such that $F_1=F^*\cap B_\lambda^* \neq \emptyset$ and, for $i>1$, it holds that $F_i = F^* \cap B_\lambda^{\dot{\bigcup}_{j=1}^{i-1} F_j} \neq \emptyset$.
\end{lemma}
\begin{proof}
    By \cref{lem: at least one element from B^*}, we know that $F_1=F^* \cap B_\lambda^* \neq \emptyset$.
    If $\vert F_1 \vert =\ell$, we are done.
    Let $i>1$ and $F_p = F^* \cap B_\lambda^{\dot{\bigcup}_{j=1}^{p-1} F_j} \neq \emptyset$ for all $p=1,\dotsc,i$.
    We have to show that $F_{i+1} = F^* \cap B_\lambda^{\dot{\bigcup}_{j=1}^{i} F_j}\neq \emptyset$ if $\lvert {\dot{\bigcup}_{j=1}^{i} F_j}\rvert <\ell $. Otherwise, we are done.
    So let $\lvert {\dot{\bigcup}_{j=1}^{i} F_j}\rvert <\ell $ and suppose that $F_{i+1} = \emptyset$.

    If we now interdict the remaining $\ell-\lvert {\dot{\bigcup}_{j=1}^{i} F_j}\rvert$ elements from the set $F^*$, the basis $B_\lambda^{\dot{\bigcup}_{j=1}^{i} F_j}$ remains, so that $w(B_\lambda^{F^*},\lambda) \leq w(B_\lambda^{\dot{\bigcup}_{j=1}^{i} F_j},\lambda)$.
    Furthermore, since $\dot{\bigcup}_{j=1}^{i} F_j\subseteq F^*$, it follows that $w(B_\lambda^{\dot{\bigcup}_{j=1}^{i} F_j},\lambda) \leq w(B_\lambda^{F^*},\lambda)$ and, therefore, $y_{F^*}(\lambda)=w(B_\lambda^{F^*},\lambda) = w(B_\lambda^{\dot{\bigcup}_{j=1}^{i} F_j},\lambda)$.
    However, if we delete an element $e\in B_\lambda^{\dot{\bigcup}_{j=1}^{i} F_j}$, then it holds that $$y_{\left(\dot{\bigcup}_{j=1}^{i} F_j\right) + e}(\lambda) > y_{\dot{\bigcup}_{j=1}^{i} F_j}(\lambda) = y_{F^*}(\lambda)$$ due to uniqueness of $B_\lambda^{\dot{\bigcup}_{j=1}^{i} F_j}$, see Assumption \ref{ass: c}.
    This is a contradiction to the optimality of $F^*$ at $\lambda$ such that $F_{i+1}\neq\emptyset$ and the claim follows by induction.
    Note that the sets $F_i$ are pairwise disjoint by construction.
\end{proof}
\begin{corollary}\label{cor: at most k+l choose l subsets}
    There are at most $k\binom{k+\ell-2}{\ell-1}$ many non-dominated sets $F\subseteq E$ with $\vert F \vert =\ell$ on $J$.
\end{corollary}
\begin{proof}
    First, we count the number of sets of the form $F^*$ from \cref{lem: F* is union of Fi} for a fixed $\lambda\in J$.
    Let $\vert F_1 \vert = j$.
    Each element $e\in F_1$ leads to a chain of the form $e_1\coloneqq e,e_2,\dotsc,e_s$ with $s \leq\ell -j+1$ where $e_p$ is the replacement element of $e_{p-1}$ with respect to the basis $B_\lambda^{\dot{\bigcup}_{j=1}^{p-2} F_j}$.
    Here, we set $B_\lambda^\emptyset \coloneqq B_\lambda^*$.
    If we have selected $j$ elements from $B_\lambda^*$ for the set $F_1$, we can choose the remaining $\ell-j$ elements from the $j$ chains for $e\in F_1$.
    However, an element $e_p$ can only be chosen if the elements $e_q$ for $q<p$ are chosen as well.
    This corresponds to the case of drawing $\ell-j$ elements from a set of $j$ elements with replacement, while disregarding the order of the draws.
    This yields $\binom{j+\ell-j+1}{\ell-j}=\binom{\ell-1}{\ell-j}$ many possibilities for a given $j$.
    Since there are $\binom{k}{j}$ many possible sets $F_1$, using Vandermonde's identity we obtain a total of $\sum_{j=1}^{\ell} \binom{k}{j} \binom{\ell-1}{\ell-j} = \binom{k+\ell-1}{\ell}$ many subsets at $\lambda$.
    This bound holds for the whole interval $J$ as the optimal basis $B_\lambda^*$ and all replacement elements $e_p$ for $e\in B_\lambda^*$ and $p=1,\dotsc,s$ remain unchanged on $J$.

    Now, we additionally want to use \cref{lem: l-1 elements chosen}.
    For $\ell-1$, the above argumentation provides exactly $\binom{k+\ell-2}{\ell-1}$ many subsets $F'\subseteq E$ with $\vert F'\vert=\ell-1$.
    To determine the missing $\ell$-th element, we compute the most vital element $e^*$ on $\M_{F'}$ for each of these subsets $F'$.
    Then, $F'+e^*$ is the only relevant subset containing $F'$.
    According to \cref{lem: at least one element from B^*}, there are at most $k$ different most vital elements on $\M_{F'}$ on the subinterval $J$.
    This yields a total of $k\binom{k+\ell-2}{\ell-1}$ many non-dominated subsets on $J$.
\end{proof}
Analogously to the proof of \cref{cor: bound using U}, we can deduce a smaller bound on the number of changepoints from \cref{cor: at most k+l choose l subsets}.
\begin{corollary}
    The optimal interdiction value function $y$ of the parametric matroid $\ell$-interdiction problem has at most $\binom{m}{2}\binom{k+\ell-2}{\ell-1}k \in \O{m^2(k+\ell)^{\ell-1}k}$ many changepoints.
\end{corollary}

In the following, we show how the sets $F\subseteq E$ from \cref{cor: at most k+l choose l subsets} can be computed on $J$.
To do so, we extend the algorithm of \cite{bazgan2011efficient} for non-parametric minimum spanning trees to parametric matroids.
First, we generalize Lemma~3 of \cite{bazgan2011efficient}.
Recall the bases $B_\lambda^i$ from \cref{def: bases Bi of Ui}.
\begin{lemma}\label{lem: r(e) in Bi+1 für e in Bi}
    Let $\lambda\in J$.
    For each element $e\in B_\lambda^i$, the replacement element $r_\lambda(e)$ with respect to $B_\lambda^i$ is contained in $B_\lambda^{i+1}$ for each $i=0,\dotsc,\ell-1$.
\end{lemma}
\begin{proof}
    Consider $g\in B_\lambda^0$.
    We remove $g$ and replace it by $r_\lambda(g)$ to obtain $B_\lambda^g=B_\lambda^*-g+r_\lambda(g)$.
    By \cref{obs: deleting F in arbitrary order}, we can iteratively remove and replace all elements from $B_\lambda^0-g$ from $B_\lambda^0$ to obtain $B_\lambda^{B_\lambda^0}=B_\lambda^1$.
    Hence, we have $r_\lambda(g)\in B_\lambda^1$ and the claim follows by induction.
\end{proof}
The algorithm computes all subsets $F\subseteq E$ from \cref{cor: at most k+l choose l subsets} and the corresponding objective function $y_F$ on $J$.
A search tree of depth $\ell-1$ is constructed to obtain these subsets.
The following objects belong to a node $s$ of level $i$.
\begin{itemize}[noitemsep]
    \item $F(s)$ is a subset of $i$ elements belonging to a tentative subset of the $\ell$-most vital elements.
    \item $\Tilde{U}(s) = \bigcup_{j=0}^{\ell-i} T_j(s)$ where $T_0(s)=B_\lambda^{F(s)}$ and for $p>0$, $T_p(s)$ is the optimal basis of the matroid $\M_{F(s)\cup\bigcup_{j=0}^{p-1} T_j(s)}$ where the set $F(s)$ and the bases $T_0(s),\dotsc,T_{p-1}(s)$ are already removed.
    \item $f(s)$ is a subset of $T_0(s)$ of elements forbidden to delete. These elements belong to every optimal basis of descendants of $s$.
    The cardinality $\vert f(s)\vert$ varies between 0 and $k-1$ depending on the position of $s$ in the search tree.
\end{itemize}
Later, we call the algorithm only for a fixed point $\lambda\in J$ and, therefore, do not refer to $\lambda$ in our notation.
For $i=0,\dotsc,\ell-1$, let $N_i$ be the set of nodes at level $i$.
The root $a$ of the search tree is initialized by setting
$$F(s)=f(s)=\emptyset, \; \Tilde{U}(a)=U_\lambda^\ell, \; w(T_0(a),\lambda)=w(T_0,\lambda), \; \text{and} \; N_0=\Set{a}.$$
Then, we need to determine the replacement elements $r_\lambda(e)$ of all elements $e$ of the corresponding basis $T_0(s)$ for each $s\in N_i$ at each level $i=0,\dotsc,\ell-1$.
At node $s$, the elements allowed to be interdicted are $T_0(s)\setminus f(s)$.
In order to avoid the same results twice on two different paths of the search tree, we enumerate the elements of $T_0(s)\setminus f(s)=\Set{e_1,\dotsc,e_{k-\vert f(s)\vert}}$.
Each of these elements $e_j$ provides a child $d$ of $s$ whose corresponding objects are computed as follows.
\begin{itemize}[noitemsep]
    \item $F(d)=F(s)+e_j$
    \item $f(d)=f(s) \cup (\bigcup_{q = 1}^{j-1} e_q)$
    \item $\Tilde{U}(d)$ is obtained from $\Tilde{U}(d)$ as follows
            \begin{itemize}[noitemsep]
                \item $T_0(d)=T_0(s)-e_j+r_\lambda(e_j)$
                \item $T_j(d)$ is derived from $T_j(s)$ for $j=1,\dotsc,\ell-\vert F(s)\vert$ as follows.
                First, delete the replacement element $r_\lambda(e)$ with respect to $T_{j-1}$ from $T_j(s)$ where $e$ is the element deleted from $T_{j-1}(s)$.
                Then, replace $g=r_\lambda(e)$ by its replacement element $r_\lambda(g)$ with respect to $T_j(s)$ where $r_\lambda(g)\in T_{j+1}(s)$ by \cref{lem: r(e) in Bi+1 für e in Bi}.
            \end{itemize}
\end{itemize}
Note that the construction of all the bases $T_j(d)$ leads exactly to the chains of replacement elements from  \cref{cor: at most k+l choose l subsets}.
At the last level $\ell-1$, we can delete one more element from the basis $T_0(s)$ at each node $s\in N_{\ell-1}$.
By \cref{lem: l-1 elements chosen}, it suffices to consider $F(s)+e^*$ where $e^*$ is the most vital element on $\M_{F(s)}$.
We later run the algorithm only for a fixed $\lambda$ between two consecutive equality points and must therefore take into account that the most vital element can change on such an interval.
According to \cref{lem: at least one element from B^*}, for each $s\in N_{\ell-1}$, we obtain the $k$ relevant subsets $F=F(s)+e$ for each $e\in B_\lambda^{F(s)}$.
The corresponding function value $y_F(\lambda)$ computes as $y_F(\lambda)=w(T_0(s),\lambda)-w(e,\lambda)+w(r_\lambda(e),\lambda)$.

The procedure is summarized in \cref{alg: compute F_i}.
There are three differences to the algorithm for minimum spanning trees.
Trivially, the replacement edges are exchanged for replacement elements.
The second difference lies in the computation of the replacement elements to determine the set $\tilde{U}(d)$. The details are given in the proof of \cref{theo: correctness alg Fi}.
Third, due to the parametric setting, we do not compute the set of $\ell$-most vital edges $F^*$ at $\lambda$ but relevant candidates $F$ and the corresponding function $y_F$.
This is because we later compute the upper envelope of these functions to obtain a solution of \cref{prob: l-matroid} on a whole subinterval of $I$.

\begin{algorithm}[!tp] 
    \Input{A matroid \M{} with weights $w(e,\lambda)$ and a fixed point~$\lambda\in J$.} 
    \Output{A set of candidates $\mathcal{C}$ for the $\ell$-most vital elements and the functions $y_F$ for $F\in\mathcal{C}$ at $\lambda$.}
    
    Compute $U_\lambda^\ell$\;
    Let $a$ be the root of the search tree\;
    Set $F(a) \algass \emptyset$, $f(a)\algass \emptyset$, $w(T_0(a),\lambda)\algass w(T_0,\lambda)$ and $\Tilde{U}(a)\algass U_\lambda^\ell$\;
    Set $N_0\algass \Set{a}$ and $N_i\algass\emptyset$ for $i=1,\dotsc,\ell-1$\;
    \For{$i = 0, \dotsc, \ell-2$}{
    	\For{$s\in N_i$}{
    		\For{$e\in T_0(s)$}{
    			Determine $r_\lambda(e)$ which is contained in $T_1(s)$\;
    		}
    		\For{$e_j\in T_0(s)\setminus f(s)$}{
    			Create child $d$ of $s$\;
    			Set $F(d)\algass F(s)+e_j$\;
    			Set $f(s) \algass f(s) \cup ( \bigcup_{q = 1}^{j-1} e_q)$\;
    			Compute $w(T_0(d),\lambda)\algass w(T_0(s),\lambda)-w(e_j,\lambda)+w(r_\lambda(e_j),\lambda)$\;
    			Determine $\Tilde{U}(d)$\;
    			Set $N_{i+1} \algass N_{i+1}+d$
    		}
    	}
    }
    Set $\mathcal{C}\algass\emptyset$\;
    \For{$s\in N_{\ell-1}$}{
    	\For{$e\in T_0(s)$}{
    		Determine $r_\lambda(e)$ which is contained in $T_1(s)$\;
    		Set $F\algass F(s)+e$\;
    		Set $\mathcal{C}\algass\mathcal{C}\cup F$\;
    		Compute $y_F(\lambda)=w(T_0(s),\lambda)-w(e,\lambda)+w(r_\lambda(e),\lambda)$\;
    	}
    }
    \Return Set $\mathcal{C}$ with corresponding functions $y_{F}$ for $F\in\mathcal{C}$\;

    \caption{An algorithm for computing the $k\binom{k+\ell-2}{\ell-1}$ relevant candidates $F\subseteq E$ for $F^*$ with corresponding objective $y_F$ on $J$. 
    \label{alg: compute F_i}}
\end{algorithm}

\begin{theorem}\label{theo: correctness alg Fi}
    For a fixed value $\lambda\in I$, \cref{alg: compute F_i} computes the $k\binom{k+\ell-2}{\ell-1}$ relevant subsets $F\subseteq E$ with corresponding function $y_F$ in time
    \begin{displaymath}
    \mathcal{O}\Big((k+\ell)^{\ell-1}\big(H(\m)+k\ell f(m) \big) + \m\big(\log m +\ell f(\m)\big)\Big).
    \end{displaymath}
\end{theorem}
\begin{proof}
    The correctness follows analogously to the correctness of the algorithm of \cite{bazgan2011efficient}.
    Similarly, we obtain the following equation for the running time
    $$t_u + \sum_{i=0}^{\ell-1}\vert N_i\vert t_{rep} + \sum_{i=1}^{\ell-1}\vert N_i\vert t_{gen} + \vert N_{\ell-1} \vert k,$$
    where we denote by $t_u$ the time needed to compute $U_\lambda^\ell$, by $t_{rep}$ the time needed to compute the replacement elements of all elements of a given basis, and by $t_{gen}$ the time for generating a node $s$ of the search tree.
    Analogously, the number $\vert N_i \vert$ of nodes at level $i$ is given by $\vert N_i \vert = \binom{k+i-1}{i}$.
    The time $t_u$ is in $\O{m\ell f(m)}$ and in $\O{m\log m +m\ell f(m)}$ if the elements have to be sorted first.
    The time $t_{rep}$ is in $\O{H(m)}$.
    To generate a node $s$ of the search tree we have to determine $F(s),f(s)$, and $\tilde{U}(s)$.
    The sets $F(s)$ and $f(s)$ can be computed in $\O{1}$ and $\O{\ell}$ time, respectively.
    Let $e \in T_{j-1}$ and $g=r_\lambda(e)$ with respect to $T_{j-1}(s)$ for some $j\in\Set{1,\dotsc,\ell-\vert F(s)\vert}$.
    To compute the element $r_\lambda(g)$, it is sufficient to test all elements of $T_{j+1}(s)$ according to \cref{lem: r(e) in Bi+1 für e in Bi}.
    This results in $k$ independence tests.
    Thus, at level $i$, we get $(\ell-i+1)k$ many independence tests and, therefore, $t_{gen}$ is in $\O{k\ell f(m)}$.
    In summary, we obtain a total running time of
    \begin{align*}
        &\O{\m\log \m + \m\ell f(\m) + (k+\ell)^{\ell-1}H(\m) + (k+\ell)^{\ell-1}k\ell f(\m)} \\
        =\; &\mathcal{O}\Big((k+\ell)^{\ell-1}\big(H(\m)+k\ell f(m) \big) + \m\big(\log m +\ell f(\m)\big)\Big).
    \end{align*}
\end{proof}

\section{Algorithms}\label{sec: algorithms}
We use our results from the previous sections to develop three exact algorithms for the parametric matroid $\ell$-interdiction problem.
The algorithms run in polynomial time whenever an independence test can be performed in time polynomial in the input length.
All our algorithms compute the upper envelope $y$ of the functions $y_F$, where $F\subseteq E$ with $\vert F\vert=\ell$.
For a value $\lambda\in I$, the set $F^*$ of $\ell$-most vital elements at $\lambda$ can then directly be derived from $y(\lambda)=y_{F^*}(\lambda)$.

\textbf{Algorithm 1.}
Our first algorithm follows from \cref{obs: y is upper env continuous}.
For each subset $F\subseteq E$ of $\ell$ elements, we solve the parametric matroid problem on the matroid $\M_F$ where the set $F$ is interdicted.
As a result, we obtain all functions $y_F$.
We then use the algorithm of \cite{hershberger1989finding} to obtain the upper envelope $y$ of all functions $y_F$.
\begin{theorem}\label{theo: algorithm 1 for l}
    \cref{prob: l-matroid} can be solved in time
    \begin{displaymath}
        \mathcal{O}\big(m^{\ell+1}(mf(m)+k^{\frac{1}{3}}\ell\log (mk))\big). 
    \end{displaymath}
\end{theorem}
\begin{proof}
    First, we compute all equality points and sort them in ascending order in $\mathcal{O}(m^2\log m)$ time.
    
    Next, we compute the optimal basis $B_\lambda^F$ for all subsets $F\subseteq E$ of $\ell$ elements for a point $\lambda\in I$ before the first equality point.
    This provides us with the linear piece of the function $y_F$ before the first equality point.     For one subset $F$, this can be done using the greedy algorithm in $\O{mf(m)}$ time.
    We then iterate over the sequence of equality points and, for each of them, we check if there is a breakpoint of one of the functions $y_F$.
    This can be done at an equality point $\lambda(e\to f)$ by performing an independence test of $B_\lambda^F-e+f$ in $\O{f(m)}$ time.
    Thus, the $\binom{m}{\ell}$ functions $y_F$ can be computed in $\O{m^\ell mf(m) + m^2m^\ell f(m)}=\O{m^{\ell+2}f(m)}$ time.

    Finally, we use the algorithm of \cite{hershberger1989finding} to compute the upper envelope $y$ of the all functions $y_F$.
    This takes $\O{t\log t}$ time, where $t$ is the total number of linear segments of all functions $y_F$.
    We have $t\in\O{m^{\ell+1}k^{\frac{1}{3}}}$ as each of the functions $y_F$ hast at most $\O{mk^{\frac{1}{3}}}$ many breakpoints, cf. \cite{dey1998improved}
    As described above, we obtain the set $F^*$ of $\ell$-most vital elements at $\lambda$ directly from $y(\lambda)$.

    In summary, we obtain a total running time of $\O{m^2\log m + m^{\ell+2}f(m) + m^{\ell+1}k^{\frac{1}{3}}\ell\log(mk)} = \O{m^{\ell+1}(mf(m)+k^{\frac{1}{3}}\ell\log (mk))}$.    
\end{proof}
Our next two algorithms use the structural properties from \cref{sec: structure}. Accordingly, the optimal interdiction value function is determined section by section between two consecutive equality points.

\textbf{Algorithm 2.}
For our second method, we use \cref{lem: F* subset U l+1} algorithmically.
To do so, the set $U_\lambda^{\ell-1}$ must be known for each subinterval between two consecutive equality points.
We use the functions $y_F$ for all $F\subseteq U_\lambda^{\ell-1}$ with $\vert F\vert=\ell$ for the computation of $y$.
We now show how the set $U_\lambda^{\ell-1}$ can be updated at an equality point $\lambda(e\to f)$.
The update step is summarized in \cref{alg: updateU}.
\begin{theorem}\label{theo: update U^l}
    Let $\lambda'\coloneqq\lambda(e\to f)$ be an equality point, $U_1 = U_\lambda^{\ell-1}$ for $\lambda \leq \lambda'$, and $U_2 = U_\lambda^{\ell-1}$ for $\lambda > \lambda'$.
    Then, we can compute $U_2$ from $U_1$ with \cref{alg: updateU} in $\mathcal{O}(f(m))$ time.
\end{theorem}
\begin{proof}
    Let $U_1=\bigcup_{j=0}^{\ell-1} B_1^j$ and $U_2=\bigcup_{j=0}^{\ell-1} B_2^j$.
    We consider four cases and distinguish whether $e$ or $f$ is an element of $U_1$.

    \textbf{Case 1:} Let $e,f \notin U_1$ and suppose that $e\in B_2^j$ for some $j$.
    As only $e$ and $f$ change their sorting, the greedy algorithm computes $B_2^i = B_1^i$ for all $i<j$.
    Now, there exists some element $g\in B_1^j$ such that $B_2^j=B_1^j-e+g$ and $w(e,\lambda')=w(g,\lambda')$ but this is a contradiction as $f\notin U_1$ and, hence, $g\neq f$.
    The proof for $f$ can be done analogously such that $e,f \notin U_2$ and no update is needed.

    \textbf{Case 2:} Let $f\in U_1$ and $e\notin U_1$.
    Suppose that $e \in U_2$, then we get a contradiction with the same argument as in Case 1.
    Now, we have $f \in B_1^j$ for some $j$.
    Again, the greedy algorithm computes $B_2^i = B_1^i$ for all $i<j$.
    In iteration $j$, the greedy algorithm tests $f$ before $e$ and chooses $f$ for $B_2^j$, since otherwise there exists a circuit in $B_2^j$ and, hence, in $B_1^j$ not containing $e$ but $f$.
    Therefore, $f \in U_2$ and no update is needed.

    \textbf{Case 3:} Let $e\in U_1$ and $f\notin U_1$.
    Then, $e \in B_1^j$ for some $j$ and $B_2^i = B_1^i$ for all $i<j$.
    Now, we consider iteration $j$ and apply the greedy algorithm before and after $\lambda'$.
    Let $e_1,\dotsc,e_r$ be the elements in $\M_{\bigcup_{i=0}^{j-1} B_1^i}$.
    Let $F_q$ and $F_q'$ be the independent sets obtained after iteration $q$ before and after $\lambda'$, respectively.
    Before $\lambda'$, it holds that $e=e_p$ and $f=e_{p+1}$ for some $p\in\Set{1,\dotsc,r}$, and after $\lambda'$, we have $f=e_p$ and $e=e_{p+1}$.
    For $q<p$, we get $F_q=F_q'$ and in iteration $p$ we have $F_p=F_{p-1}+e$.
    Now, we distinguish the two cases whether $F_{p-1}'+f$ is independent or not.
    
    \textbf{3a:} If $F_{p-1}'+f \notin \F$, we obtain $F_p'+e = F_{p-1}'+e = F_{p-1}+e \in \F$ and, hence, $e\in B_2^j$ such that $B_2^i = B_1^i$ for all $i=j,\dotsc,\ell-1$.

    \textbf{3b:} If $F_{p-1}'+f \in \F$, we obtain $F_p'=F_{p-1}'+f = F_{p-1}+f=F_p-e+f$.
    We show inductively that $F_q'= F_q-e+f$ is maintained for all $q>p$.
    Note that $F_p'+e$ and $F_p+f$ contain a circuit $C$.
    Let $q>p$ and $F_{q-1}'= F_{q-1}-e+f$. We show that $F_{q-1} + e_q \in \F$ if and only if $F_{q-1}' +e_q\in\F$.
    If $F_{q-1}+e_q \in\F$, then $\lvert F_{q-1}'\rvert = \lvert F_{q-1}+e_q\rvert -1$ and, hence, there exists a $g\in F_{q-1}+e_q = F_{q-1}' -f +e +e_q$ such that $F_{q-1}' +g\in\F$.
    As $F_{i}'+e$ contains $C$, it follows that $e_q\neq e$ and, therefore, $g=e_q$.
    Analogously, it follows that $F_{q-1} + e_q \in \F$ if $F_{q-1}' +e_q\in\F$.

    To summarize, in Case 3b, the equation $B_2^j=B_1^j -e+f$ holds, so that according to \cref{lem: r(e) = f before swap}, before the point $\lambda'$, the element $f$ is the replacement element of $e$ with respect to the basis $B_1^j$.
    Thus, $f \in B_1^{j+1}$ by \cref{lem: r(e) in Bi+1 für e in Bi} and according to the condition from Case 3, $j$ equals $\ell-1$ and no basis with an index higher than $j$ needs to be updated.
    Hence, $U_2=U_1-e+f$.
    In Case 3a, no update is required.
    In total, we can update the set $U_1$ in Case 3 by performing a single independence test in $\mathcal{O}(f(m))$ time.

    \textbf{Case 4:}
    Let $e,f\in U_1$. Then, $e\in B_1^j$ and $f\in B_1^s$ for some $j,s\in\Set{0,\dotsc,\ell-1}$.

    \textbf{4a:} If $j=s$, the basis $B_1^j$ and, hence, all bases $B_1^i$ for $i>j$ remain the same after $\lambda'$ such that no update is needed.

    \textbf{4b:} If $s<j$, then before $\lambda'$, $e$ gets tested before $f$ and is rejected.
    Instead, $f$ is chosen so that we get the same bases $B_2^i=B_1^i$ for all $i$ after $\lambda'$, as $f$ is now tested first and chosen again before $e$.

    \textbf{4c:} If $j<s$, we get $B_1^i = B_2^i$ for all $i<j$.
    Before $\lambda'$, the element $e$ is tested and selected first and then $f$ is rejected.
    After $\lambda'$, the element $f$ is tested first and, as in Case 3, we get two cases, depending on whether the independence test of $f$ detects a circuit or not.
    In Case 3a, no update is required and, in Case 3b, we have $B_2^j=B_1^j -e+f$.
    By \cref{lem: r(e) in Bi+1 für e in Bi} it holds that $s=j+1$ and we show that $B_2^{j+1}=B_1^{j+1} -f+e$.
    Then, after iteration $j+1$ before and after $\lambda'$, the same elements were deleted, namely $\bigcup_{i=0}^{j+1} B_2^i = \left(\bigcup_{i=0}^{j-1} B_1^i \right) \cup \left(B_1^j-e+f \right) \cup \left(B_1^{j+1}-f+e\right) = \bigcup_{i=0}^{j+1} B_1^i$ so that we do not need an update, since $e,f\in U_2$ remain.

    We consider iteration $j+1$ and let again $F_q$ and $F_q'$ be the independent sets obtained after iteration $q$ before and after $\lambda'$, respectively.
    Before $\lambda'$, we have $f\in B_1^{j+1}$ and, hence, $f\in F_p$ for some $p$.
    Then, $F_q=F_q'$ for $q<p$ and $F_p=F_{p-1}+f$.
    Since $e=r_\lambda(f) \in B_2^{j+1}$ for $\lambda >\lambda'$ and $e$ is tested immediately after $f$, it follows that $F_p'=F_{p-1}'+e$.
    The equality $F_q'=F_q-f+e$ for all $q>j+1$ can be proven analogously to Case 3b.

    To summarize, in Case 4, we obtain $e,f\in U_2$.
    The checks whether element $e$ or $f$ is contained in a basis $B_1^j$ can be performed in constant time $\mathcal{O}(1)$.
    Overall, \cref{alg: updateU} correctly updates the set $U_1$ in $\mathcal{O}(f(m))$ time.
\end{proof}
If there is an update of $U_\lambda^{\ell-1}$ in Case 3b of \cref{theo: update U^l}, only the element $e$ is replaced by $f$.
This means that a subset $F$ of $U_1$ remains relevant, i.~e.\ it remains a subset of $U_2$, unless $U_2=U_1-e+f$ and $e\in F$.
In this case, we update $F$ to $F-e+f\subseteq U_2$.
We now show how the corresponding basis $B_\lambda^F$ can be updated.

\begin{algorithm}[!tp]
    \SetKwFunction{UpdateU}{UpdateU}
    \KwIn{An equality point $\lambda(e\to f)$, the set $U_1 = U_\lambda^{\ell-1}$ before $\lambda(e\to f)$.}
    \KwOut{The set $U_2 = U_\lambda^{\ell-1}$ after $\lambda(e\to f)$.}
    \Def{\UpdateU{$U_1$, $\lambda(e\to f)$}}{
    	Let $U_1=\bigcup_{i=0}^{\ell-1} B_1^i$\;
    	\uIf{$e\in B_1^{\ell-1}$, $f\notin U_1$ and $B_1^{\ell-1}-e+f$ independent}
    	{
    		$U_2 \algass U_1-e+f$\;
    	}
    	\Else
    	{
    		$U_2 \algass U_1$\;        
    	}
    }
    \caption{An algorithm for updating $U_\lambda^{\ell-1}$ at an equality point.
    \label{alg: updateU}}
\end{algorithm}

\begin{theorem}\label{theo: update B^F for F in U^l}
    Let $\lambda'\coloneqq\lambda(e\to f)$ be an equality point, $U_1 = U_\lambda^{\ell-1}$ for $\lambda \leq \lambda'$, $U_2 = U_\lambda^{\ell-1}$ for $\lambda > \lambda(e\to f)'$ and $F\subseteq U_1$.
    Then, we can update the set $F$ and the corresponding basis $B_\lambda^F$ for $\lambda> \lambda'$ in $\mathcal{O}(f(m))$ time.
\end{theorem}
\begin{proof}
    We define $F_1\coloneqq F$ and use $F_2$ to denote the potentially updated set $F_1$ after $\lambda'$.
    Accordingly, for $\lambda\leq\lambda'$, we define $B_1^F\coloneqq B_\lambda^{F_1}$ and for $\lambda>\lambda'$, we define $B_2^F\coloneqq B_\lambda^{F_2}$.
    We show how the basis $B_2^F$ can be derived from $B_1^F$ .
    
    First, we consider the case $U_2=U_1$.
    Then, $F_1=F_2$.
    We require the optimal basis on $\M_{F_1}$ before and after $\lambda'$, so we only need to check whether $B_1^F-e+f$ is independent.
    This can be done in $\mathcal{O}(f(m))$ time and if the answer is yes, we update $B_2^F=B_1^F-e+f$.
    The same applies if $U_1$ is updated to $U_2=U_1-e+f$, but $e$ is not an element of $F_1$.

    Now, we consider the case that $U_2=U_1-e+f$ and $e\in F_1$.
    Then, $F_1 = X+e$ and $F_2=X+f$.
    We know that $f\notin B_1^X\coloneqq B_\lambda^X$ for $\lambda\leq \lambda'$ and $e\notin B_2^X\coloneqq B_\lambda^X$ for $\lambda > \lambda'$ since $f\notin U_1$ and $e\notin U_2$.

    \textbf{Case 1:} If $e\notin B_1^X$, then $e,f\notin B_1^X$ such that there is no swap between $e$ and $f$ at $\lambda'$ in $B_\lambda^X$ and it follows that $B_1^X=B_2^X$.
    In this case, we obtain $B_2^F=B_2^X=B_1^X=B_1^F$.

    \textbf{Case 2:} If $e\in B_1^X$, there exists a $g\notin B_1^X$ such that $B_2^X=B_1^X-e+g$ since $e\notin B_2^X$.
    But then $w(e,\lambda')=w(g,\lambda')$ such that $g=f$ and $B_2^X=B_1^X-e+f$.
    By \cref{lem: r(e) = f before swap}, we know that $f=r_\lambda(e)$ with respect to basis $B_1^X$ before $\lambda'$ and $e=r_\lambda(f)$ with respect to basis $B_2^X$ after $\lambda'$.
    We obtain $B_1^F=B_1^X-e+f$ and $B_2^F=B_2^X-f+e$.
    In Case 2, it follows that $B_2^F=B_2^X-f+e =B_1^X=B_1^F+e-f$.

    However, only the basis $B_1^F$ is known before $\lambda'$ and the computation of $B_1^X$ can in general not be done in $\mathcal{O}(f(m))$ time.
    In order to keep the promised running time, we show that Case 1 only occurs if $f\notin B_1^F$ and Case 2 only occurs if $f\in B_1^F$.
    These tests can be performed in constant time $\O{1}$.
    Let $f\notin B_1^F$ and suppose that $e\in B_1^X$. It follows from Case 2 that $r_\lambda(e)=f \in B_1^F$ which is a contradiction.
    Further, if $f\in B_1^F$ and $e\notin B_1^X$, we get a contradiction since $f\notin B_2^F$ but $B_2^F=B_1^F$ by Case 1.
    To summarize, we can update the basis $B_1^F$ and the set $F_1$ in $\O{f(m)}$ time.
\end{proof}
Note that updating $B_\lambda^F$ with \cref{theo: update B^F for F in U^l} is faster by a factor of $m$ than recalculating it with the greedy algorithm.
\cref{theo: update B^F for F in U^l} implies that it suffices to consider $\binom{k\ell}{\ell}$ many piecewise linear and continuous functions on the entire parameter interval $I$ for the computation of $y$.
\begin{corollary}
    The optimal interdiction value function $y$ of the parametric matroid $\ell$-interdiction problem has at most $\mathcal{O}(\m^{\ell+1} k^{\frac{1}{3}}\alpha(k\ell))$ many changepoints.
\end{corollary}
\begin{proof}
    Let $\lambda_i=\lambda(e\to f)$ be an equality point and $\lambda_{i-1}$ and $\lambda_{i+1}$ be the next smaller and larger equality point of $\lambda_i$, respectively.
    Let $I_1=(\lambda_{i-1},\lambda_i]$ and $I_2 =(\lambda_i,\lambda_{i+1}]$.
    Let $U_1=U_\lambda^{\ell-1}$ for $\lambda \in I_1$ and $U_2=U_\lambda^{\ell-1}$ for $\lambda \in I_2$.
    Let $F_1\subseteq U_1$. We use $F_2$ to denote the potentially updated set $F_1$ on $I_2$.
    Finally, let $B_1^F\coloneqq B_\lambda^{F_1}$ for $\lambda\in I_1$ and $B_2^F\coloneqq B_\lambda^{F_2}$ for $\lambda\in I_2$.
    
    By \cref{theo: update B^F for F in U^l}, we either have $B_2^F= B_1^F$ or $B_2^F=B_1^F-e+f$ or $B_2^F=B_1^F+e-f$.
    Since $w(e,\lambda_i)=w(f,\lambda_i)$ it follows that $y_{F_1}(\lambda_i)=w(B_1^F,\lambda_i)= w(B_2^F,\lambda_i)=y_{F_2}(\lambda_i)$ such that the function defined by
    \begin{displaymath}
        \lambda \mapsto
        \begin{cases}
            y_{F_1}(\lambda),  & \text{for} \, \lambda \in I_1 \\
            y_{F_2}(\lambda), & \text{for} \, \lambda \in I_2
        \end{cases}
    \end{displaymath}
    is continuous on $(\lambda_{i-1},\lambda_{i+1}]$.

    By \cref{theo: update B^F for F in U^l}, before the first equality point $\lambda_0$, only the subsets $F\subseteq U_0=U_\lambda^{\ell-1}$ for $\lambda\leq\lambda_0$ are relevant.
    Thus, it follows inductively from the above argument that each of these subsets $F$ provides one piecewise linear and continuous function for the entire parameter interval $I$.
    Again, it follows from \cref{theo: update B^F for F in U^l} that we only need to consider these $\binom{k\ell}{\ell}$ many functions for the computation of $y$.
    Thus, it follows from \cite{agarwal1995davenport} that $y$ has at most $\mathcal{O}(\m^{\ell+1} k^{\frac{1}{3}}\alpha(k^\ell\ell^\ell))\in \mathcal{O}(\m^{\ell+1} k^{\frac{1}{3}}\alpha(k\ell))$ many changepoints.    
\end{proof}

We are now ready to state our second algorithm that computes the upper envelope $y$ on each subinterval between two consecutive equality points using \cref{theo: update B^F for F in U^l}.
The algorithm is stated in \cref{alg: compute y using U_i}

\begin{algorithm}[!tp] 
    \Input{A matroid \M{} with edge weights $w(e,\lambda)$ and parameter interval~$I$.} 
    \Output{A representation of the upper envelope $y$.}
    
    Compute all equality points $\lambda_1,\ldots,\lambda_r$ and sort them ascendingly\;
    Let $I_0,\ldots,I_r$ be the corresponding intervals\;
    Compute $U_0$ and optimal bases $B_0^F$ for $F\subseteq U_0$ on $I_0$\;
    Obtain $y^F$ on $I_0$ using $B_0^F$ and compute $y$ on $I_0$\;
    \For{$i = 1, \dotsc, r$}{
    	Update $U_{i-1}$ using \cref{alg: updateU}\;
    	\For{$F\subseteq U_{i-1}$}{
    		\uIf{$U_i=U_{i-1}-e+f$, $e\in F$ and $f\in B_{i-1}^F$}{
    			$F\algass F-e+f$\;
    			$B_i^F\algass B_{i-1}^F+e-f$
    		}
    		\uElseIf{$U_i=U_{i-1}-e+f$, $e\in F$ and $f\notin B_{i-1}^F$}{
    			$F\algass F-e+f$\;
    			$B_i^F\algass B_{i-1}^F$\;
    		}
    		\Else{
    			\uIf{$B_{i-1}^F-e+f$ is independent}{
    				$B_i^F\algass B_{i-1}^F -e+f$\;
    			}
    			\Else{
    				$B_i^F\algass B_{i-1}^F$\;
    			}
    		}
    		Obtain $y^F$ on $I_i$ using $B_i^F$\;
    	}
    	Compute $y$ on $I_i$\;
    }
    \Return $y$\;
    \caption{An algorithm for computing the optimal interdiction value function~$y$ using \cref{theo: update B^F for F in U^l}.
    \label{alg: compute y using U_i}}
\end{algorithm}

\begin{theorem} \label{theo: algorithm using U}
    \cref{prob: l-matroid} can be solved in time
    \begin{displaymath}
        \mathcal{O} \Big(\m^2\big(k^\ell\ell^\ell(f(m)+\log k\ell)+\log m \big)\Big).
    \end{displaymath}
\end{theorem}
\begin{proof}
    First, we compute all equality points $\lambda_1,\ldots,\lambda_r$ and sort them ascendingly in $\O{m^2\log m}$ time.
    We define the intervals of consecutive equality points $I_0\coloneqq (-\infty,\lambda_1]$, $I_i\coloneqq [\lambda_i,\lambda_{i+1}]$ for $1\leq i < r$, and $I_r\coloneqq [\lambda_r,\infty)$.
    For each $i=0,\dotsc,r$, let $U_i\coloneqq U_\lambda^{\ell-1}$.
    Further, for each $i=1,\dotsc,r$ and $F\subseteq U_{i-1}$, let $B_{i-1}^F=B_\lambda^F$ for $\lambda\in I_{i-1}$.
    At the equality point $\lambda_i$, we then obtain the potentially updated basis $B_i^F$ for the potentially updated set $F\subseteq U_i$.
    
    Before the first equality point $\lambda_1$, we compute the set $U_0$ in $\O{m\log m +m\ell f(m)}$ time.
    Afterwards, we compute optimal bases $B_0$ for each subset $F\subseteq U_0$ for a fixed value $\lambda\in I_0$.
    This requires $\binom{k\ell}{\ell}$ many calls of the greedy algorithm which yields a running time of $\O{k^\ell\ell^\ell mf(m)}$.
    Thus, the initialization step takes $\O{m^2\log m + k^\ell\ell^\ell mf(m)}$ time.

    Next, we compute the upper envelope $y$ separately on each of the $\O{m^2}$ intervals $I_0,\dotsc,I_r$.
    At equality point $\lambda_i$, we update the set $U_\lambda^{\ell-1}$ from $U_{i-1}$ to $U_i$ in $\O{f(m)}$ time using \cref{alg: updateU}.
    Next, we update the subsets $F \subseteq U_{i-1}$ and the corresponding bases $B_{i-1}^F$, each in $\O{f(m)}$ time using \cref{theo: update B^F for F in U^l}.
    Thus, for each interval $I_i$, each of the linear functions $y_F$ with $F\subseteq U_i$ is known.
    Their upper envelope $y$ can be computed on one interval using the algorithm of \cite{hershberger1989finding} in $\O{k^\ell\ell^\ell\log k^\ell\ell^\ell} =\O{k^\ell\ell^{\ell+1}\log k\ell}$ time.
    In summary, the running time amounts to
    \begin{align*}
        &\mathcal{O}\Big(m^2\log m + k^\ell\ell^\ell mf(m) + m^2\big(f(m) + k^\ell\ell^\ell f(m) + k^\ell\ell^{\ell+1} \log k\ell\big)\Big) \\
        =\; &\mathcal{O} \Big(\m^2\big(k^\ell\ell^\ell(f(m)+\ell\log k\ell)+\log m \big)\Big).
    \end{align*}
\end{proof}

\textbf{Algorithm 3.}
Our third algorithm is based on \cref{lem: F* is union of Fi} and \cref{theo: correctness alg Fi}.
We again consider the subintervals between to consecutive equality points separately.
Using \cref{alg: compute F_i}, we obtain the relevant functions $y_F$ on each such subinterval in order to compute their upper envelope $y$.

\begin{theorem}
    \cref{prob: l-matroid} can be solved in time
    \begin{displaymath}
        \mathcal{O}\bigg( m^2 \Big( (k+\ell)^{\ell-1} \big( H(m) + k\ell\big(f(m)+\log(k+\ell)\big)\big) + m\ell f(m) \Big) \bigg).
    \end{displaymath}
\end{theorem}
\begin{proof}
    We apply \cref{alg: compute y using F_i} to compute the optimal interdiction value function $y$.
    First, we compute the equality points $\lambda_1,\ldots,\lambda_r$ and sort them in ascending order in $\O{m^2\log m}$ time.
    Let $I_0\coloneqq (-\infty,\lambda_1]$, $I_i\coloneqq [\lambda_i,\lambda_{i+1}]$ for $1\leq i < r$, and $I_r\coloneqq [\lambda_r,\infty)$ be the corresponding intervals.

    Next, we call \cref{alg: compute F_i} once for each interval $I_i$.
    This takes $\O{\m(\ell f(\m)+\log m) + (k+\ell)^{\ell-1}(H(\m)+k\ell f(m))}$ time for the first interval $I_0$.
    For any other interval $I_i$ with $i>0$, the sorting of the edge weights can easily be updated at an equality point such that the running time reduces to $\O{\m\ell f(\m) + (k+\ell)^{\ell-1}(H(\m)+k\ell f(m))}$.
    On one interval $I_i$, the solution provides us with $k\binom{k+\ell-2}{\ell-1} \in \O{(k+\ell)^{\ell-1}k}$ many subsets $F\subseteq E$ with corresponding function $y_F$.
    Further, the upper envelope of these functions equals the optimal interdiction value function $y$.
    We use the algorithm of \cite{hershberger1989finding} to compute the upper envelope $y$ in $\O{(k+\ell)^{\ell-1}k\log((k+\ell)^{\ell-1}k)} = \O{(k+\ell)^{\ell-1}k\ell\log(k+\ell)}$ time.
    Overall, we obtain a total running time of
    \begin{align*}
        &\mathcal{O}\bigg( m^2 \Big(\ m\ell f(m) + (k+\ell)^{\ell-1} \big(\ H(m) + k\ell f(m) \big) + (k+\ell)^{\ell-1}k\ell \log(k+\ell) \Big)\bigg) \\
        =\; &\mathcal{O}\bigg( m^2 \Big( (k+\ell)^{\ell-1} \big( H(m) + k\ell\big(f(m)+k\ell\log(k+\ell)\big)\big) + m\ell f(m) \Big) \bigg).
    \end{align*}
\end{proof}

\begin{algorithm}[!tp] 
	\Input{A matroid \M{} with edge weights $w(e,\lambda)$ and parameter interval~$I$.} 
	\Output{A representation of the upper envelope $y$.}
	
	Compute all equality points $\lambda_1,\ldots,\lambda_r$ and sort them ascendingly\;
	Let $I_0,\ldots,I_r$ be the corresponding intervals\;
	\For{$i = 0, \dotsc, r$}{
		Apply \cref{alg: compute F_i} for a fixed $\lambda\in I_i$ to obtain the functions $y_F$ on $I_i$\;
		Compute $y$ on $I_i$\;
	}
	\Return $y$\;
	\caption{An algorithm for computing the optimal interdiction value function~$y$ using \cref{theo: correctness alg Fi}.
		\label{alg: compute y using F_i}}
\end{algorithm}

\textbf{Discussion.}
We finish this section with a brief discussion of our three algorithms.
For an arbitrary matroid, the running times are in general difficult to compare, as they depend on different matroid operations.
If an independence test can be performed in polynomial time, all our algorithms have polynomial running time for constant $\ell$.
This is due to the fact that a replacement element $r_\lambda(e)$ of a basis element $e\in B_\lambda^*$ can be identified by at most $m-k$ many independence tests of $B_\lambda^*-e+r$ for all $r\in E\setminus B_\lambda^*$.
Thus, we have $h(m)\in\O{mf(m)}$ and $H(m)\in\O{mkf(m)}$.

Even though, \cref{alg: compute y using F_i} uses more structure and, thus fewer relevant functions $y_F$ on each subinterval between two consecutive equality points, its runtime is not generally faster than that of \cref{alg: compute y using U_i}.
The reason for this is the efficient update of a subset $F\subseteq U_\lambda^{\ell-1}$ and the corresponding basis $B_\lambda^F$ at an equality point in $\O{f(m)}$ time.
In particular, the set $F$ and the basis $B_\lambda^F$ are only changed by a swap of $e$ and $f$ at an equality point $\lambda(e\to f)$.
In the following example, we use graphical matroids to show why a similar update with same running time complexity can in general not be performed for a subset $F\subseteq E$ of $\ell$ elements from \cref{cor: at most k+l choose l subsets}.
For graphical matroids, we refer to the replacement elements as replacement edges.
\begin{example}
    Consider the graph with parametric edge weights from \cref{fig: update Fi}.
    Let $\ell=3$.
    For simplicity, the graph is not 4-connected.
    However, this can easily be achieved by adding three parallels to each edge whose weight is equal to a large constant $M$.
    Then, in the considered subintervals, the replacement edges are given as in \cref{fig: Fi before 4} and \cref{fig: Fi after 4}.
    The (interdicted) minimum spanning trees are marked by the thick edges.
    The edge that is removed next is shown as a dashed line.
    
    For $\lambda\in [2,4)$, the minimum spanning tree $B_\lambda^*$ is given by the edges $a,b,c,e,g$.
    Further, it holds that $r_\lambda(g)=r$ with respect to $B_\lambda^*$ and $r_\lambda(r)=f$ with respect to $B_\lambda^{g}$ as well as $r_\lambda(f)=p$ with respect to $B_\lambda^{\Set{g,r}}$, see \cref{fig: Fi before 4}.
    Thus, for $\lambda\in [2,4)$, the set $F_1=\Set{g,r,f}$ is one of the sets from \cref{cor: at most k+l choose l subsets}.
    
    At the breakpoint $\lambda(e\to f)=4$, the edge $e$ leaves $B_\lambda^*$ and gets replaced by $f$.
    With a similar update as in \cref{alg: updateU} we exchange $f$ for $e$ in the set $F_1$ to update $F_1$ to $F_2=\Set{g,r,e}$.
    Note that this maintains the structure of the sets from \cref{cor: at most k+l choose l subsets}, as the replacement edge $r_\lambda(r)$ with respect to basis $B_\lambda^g$ changes from $f$ to $e$ at $\lambda=4$, see \cref{fig: Fi after 4}.
    However, the replacement edge $p$ of $f$ with respect to $B_\lambda^{\Set{g,r}}$ for $\lambda\in [2,4)$ differs from the replacement edge $q$ of $e$ with respect to $B_\lambda^{F_2}$ for $\lambda\in [4,6)$.
    Hence, the weights of the interdicted minimum spanning trees $B_\lambda^{F_1}$ and $B_\lambda^{F_2}$ are given by $w(B_\lambda^{F_1},\lambda)=7+2\lambda$ and $w(B_\lambda^{F_2},\lambda)=12+\lambda$.
    In particular, we have $w(B_4^{F_1},4)=15\neq 16= w(B_\lambda^{F_2},4)$ and in contrast to the update in \cref{theo: update B^F for F in U^l}, there is no continuous transition from basis $B_\lambda^{F_1}$ to $B_\lambda^{F_2}$ at $\lambda=4$.
    Consequently, the replacement edge of $e$ with respect to basis $B_\lambda^{F_2}$ must be re-computed.
    This does not result in an asymptotic improvement of the running time compared to a new call of \cref{alg: compute F_i}.
    In contrast, in \cref{alg: compute y using U_i}, the set $F_2=F_1$ remains relevant and the basis $B_\lambda^{F_2}$ equals the basis $B_\lambda^{F_2}$.
\end{example}

\begin{figure}
    \centering
        \begin{tikzpicture}[scale=1,every node/.style={scale=1,circle}]
            \graph  [nodes={draw,circle,fill=black!25,empty nodes}, edges={black!50}, counterclockwise, radius=8cm] 
            {subgraph C_n [n=5,m=3,radius=4cm,name=A]};
            \node at  ($(A 3)!.5!(A 4)$) (B){};
            \node[draw,circle,fill=black!25] at ($(B)!.5!(A 1)$) (A 6){};
            \foreach \i in {1,3,4}{
                \draw (A 6) -- (A \i);
            }
            \path[black!50] (A 1) edge node[left] {\textcolor{black}{$w(c,\lambda)=2$}} (A 2);
            \path[black!50] (A 2) edge node[left] {\textcolor{black}{$w(r,\lambda)=4$}} (A 3);
            \path[black!50] (A 3) edge node[below,transform canvas={yshift=0.5cm}] {\textcolor{black}{$w(p,\lambda)=7$}} (A 4);
            \path[black!50] (A 4) edge node[right] {\textcolor{black}{$w(b,\lambda)=1$}} (A 5);
            \path[black!50] (A 5) edge node[right] {\textcolor{black}{$w(e,\lambda)=-3+2\lambda$}} (A 1);
            \path[black!50] (A 6) edge node[left,pos=0.2] {\textcolor{black}{$w(a,\lambda)=0$}} (A 3);
            \path[black!50] (A 6) edge node[left,pos=0.8] {\textcolor{black}{$w(f,\lambda)=1+\lambda$}} (A 4);
            \path[black!50] (A 1) edge[bend left] node[left] {\textcolor{black}{$w(q,\lambda)=8$}} (A 4);
            \path[black!50] (A 1) edge node[left] {\textcolor{black}{$w(g,\lambda)=3$}} (A 6);
        \end{tikzpicture}
    \caption{A graph with parametric edge weights, where $\ell=3$ edges are allowed to be interdicted.}
    \label{fig: update Fi}
\end{figure}
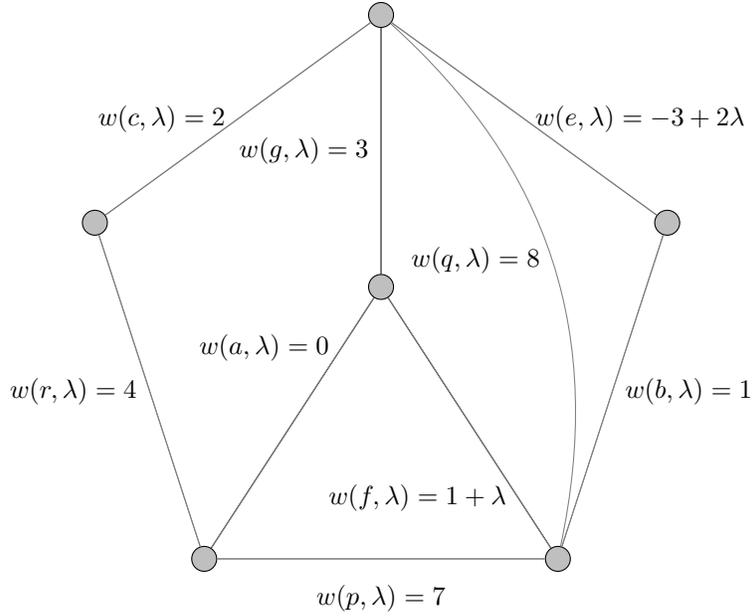

\begin{figure}
    \centering
    \begin{subfigure}[b]{0.23\textwidth}
        \centering
        \begin{tikzpicture}[scale=0.5,every node/.style={scale=1,circle}]
            \graph  [nodes={draw,circle,fill=black!25,empty nodes}, edges={black!50}, counterclockwise, radius=8cm] 
            {subgraph C_n [n=5,m=3,radius=1.3cm,name=A]};
            \node at  ($(A 3)!.5!(A 4)$) (B){};
            \node[draw,circle,fill=black!25] at ($(B)!.5!(A 1)$) (A 6){};
            \foreach \i in {3,4}{
                \draw (A 6) -- (A \i);
            }
            \path[black,line width=2] (A 1) edge (A 2);
            \path[black!50] (A 2) edge (A 3);
            \path[black!50] (A 3) edge (A 4);
            \path[black,line width=2] (A 4) edge (A 5);
            \path[black,line width=2] (A 5) edge (A 1);
            \path[dashed,black,line width=2] (A 1) edge (A 6);
            \path[black,line width=2] (A 6) edge (A 3);
            \path[black!50] (A 6) edge (A 4);
            \path[black!50] (A 1) edge[bend left] (A 4);
        \end{tikzpicture}
        \caption{$B_\lambda^*$ \\ \hspace{\textwidth}}
        \label{fig: B before 4}
    \end{subfigure}
    \hfill
    \begin{subfigure}[b]{0.23\textwidth}
        \centering
            \begin{tikzpicture}[scale=0.5,every node/.style={scale=1,circle}]
            \graph  [nodes={draw,circle,fill=black!25,empty nodes}, edges={black!50}, counterclockwise, radius=8cm] 
            {subgraph C_n [n=5,m=3,radius=1.3cm,name=A]};
            \node at  ($(A 3)!.5!(A 4)$) (B){};
            \node[draw,circle,fill=black!25] at ($(B)!.5!(A 1)$) (A 6){};
            \foreach \i in {3,4}{
                \draw (A 6) -- (A \i);
            }
            \path[black,line width=2] (A 1) edge (A 2);
            \path[white,line width=2] (A 2) edge (A 3);
            \path[dashed,black,line width=2] (A 2) edge (A 3);
            \path[black!50] (A 3) edge (A 4);
            \path[black,line width=2] (A 4) edge (A 5);
            \path[black,line width=2] (A 5) edge (A 1);
            \path[black,line width=2] (A 6) edge (A 3);
            \path[black!50] (A 6) edge (A 4);
            \path[black!50] (A 1) edge[bend left] (A 4);
        \end{tikzpicture}
        \caption{$B_\lambda^*-g+r$ \\ \hspace{\textwidth}}
        \label{fig: B-g+r before 4}
    \end{subfigure}
    \hfill
    \begin{subfigure}[b]{0.23\textwidth}
        \centering
            \begin{tikzpicture}[scale=0.5,every node/.style={scale=1,circle}]
            \graph  [nodes={draw,circle,fill=black!25,empty nodes}, edges={black!50}, counterclockwise, radius=8cm] 
            {subgraph C_n [n=5,m=3,radius=1.3cm,name=A]};
            \node at  ($(A 3)!.5!(A 4)$) (B){};
            \node[draw,circle,fill=black!25] at ($(B)!.5!(A 1)$) (A 6){};
            \foreach \i in {3,4}{
                \draw (A 6) -- (A \i);
            }
            \path[black,line width=2] (A 1) edge (A 2);
            \path[black!50,thin] (A 3) edge (A 4);
            \path[black,line width=2] (A 4) edge (A 5);
            \path[black,line width=2] (A 5) edge (A 1);
            \path[black,line width=2] (A 6) edge (A 3);
            \path[white,line width=2] (A 6) edge (A 4);
            \path[dashed,black,line width=2] (A 6) edge (A 4);
            \path[black!50] (A 1) edge[bend left] (A 4);
            \draw[white,line width=2] (A 2) -- (A 3);
        \end{tikzpicture}
        \caption{$B_\lambda^*-g+r-r+f$ \\ \hspace{\textwidth}}
        \label{fig: B-g+r-r+f before 4}
    \end{subfigure}
    \hfill
    \begin{subfigure}[b]{0.23\textwidth}
        \centering
            \begin{tikzpicture}[scale=0.5,every node/.style={scale=1,circle}]
            \graph [nodes={draw,circle,fill=black!25,empty nodes}, edges={black!50}, counterclockwise, radius=8cm] 
            {subgraph C_n [n=5,m=3,radius=1.3cm,name=A]};
            \node at  ($(A 3)!.5!(A 4)$) (B){};
            \node[draw,circle,fill=black!25] at ($(B)!.5!(A 1)$) (A 6){};
            \foreach \i in {3,4}{
                \draw (A 6) -- (A \i);
            }
            \path[black,line width=2] (A 1) edge (A 2);
            \path[white,line width=2] (A 2) edge (A 3);
            \path[black,line width=2] (A 3) edge (A 4);
            \path[black,line width=2] (A 4) edge (A 5);
            \path[black,line width=2] (A 5) edge (A 1);
            \path[black,line width=2] (A 6) edge (A 3);
            \path[white,line width=2] (A 6) edge (A 4);
            \path[black!50] (A 1) edge[bend left] (A 4);
        \end{tikzpicture}
        \caption{$B_\lambda^*-g+r-r+f-f+p$}
        \label{fig: B-g+r-r+f-f+p before 4}
    \end{subfigure}
    \caption{Before the breakpoint $\lambda(e\to f)=4$, the set $F_1=\Set{g,r,f}$ is one of the sets from \cref{cor: at most k+l choose l subsets}.
    The interdicted minimum spanning tree $B_\lambda^{F_1}$ with weight $w(B_\lambda^{F_1},\lambda)=7+2\lambda$ is given by the edges $a,b,c,e,p$, cf.\ \cref{fig: B-g+r-r+f-f+p before 4}.}
    \label{fig: Fi before 4}
\end{figure}

\begin{figure}
    \centering
    \begin{subfigure}[b]{0.23\textwidth}
        \centering
        \begin{tikzpicture}[scale=0.5,every node/.style={scale=1,circle}]
            \graph  [nodes={draw,circle,fill=black!25,empty nodes}, edges={black!50}, counterclockwise, radius=8cm] 
            {subgraph C_n [n=5,m=3,radius=1.3cm,name=A]};
            \node at  ($(A 3)!.5!(A 4)$) (B){};
            \node[draw,circle,fill=black!25] at ($(B)!.5!(A 1)$) (A 6){};
            \foreach \i in {3,4}{
                \draw (A 6) -- (A \i);
            }
            \path[black,line width=2] (A 1) edge (A 2);
            \path[black!50] (A 2) edge (A 3);
            \path[black!50] (A 3) edge (A 4);
            \path[black,line width=2] (A 4) edge (A 5);
            \path[black!50] (A 5) edge (A 1);
            \path[dashed,black,line width=2] (A 1) edge (A 6);
            \path[black,line width=2] (A 6) edge (A 3);
            \path[black,line width=2] (A 6) edge (A 4);
            \path[black!50] (A 1) edge[bend left] (A 4);
        \end{tikzpicture}
        \caption{$B_\lambda^*$ \\ \hspace{\textwidth}}
        \label{fig: B after 4}
    \end{subfigure}
    \hfill
    \begin{subfigure}[b]{0.23\textwidth}
        \centering
            \begin{tikzpicture}[scale=0.5,every node/.style={scale=1,circle}]
            \graph  [nodes={draw,circle,fill=black!25,empty nodes}, edges={black!50}, counterclockwise, radius=8cm] 
            {subgraph C_n [n=5,m=3,radius=1.3cm,name=A]};
            \node at  ($(A 3)!.5!(A 4)$) (B){};
            \node[draw,circle,fill=black!25] at ($(B)!.5!(A 1)$) (A 6){};
            \foreach \i in {3,4}{
                \draw (A 6) -- (A \i);
            }
            \path[black,line width=2] (A 1) edge (A 2);
            \path[white,line width=2] (A 2) edge (A 3);
            \path[dashed,black,line width=2] (A 2) edge (A 3);
            \path[black!50] (A 3) edge (A 4);
            \path[black,line width=2] (A 4) edge (A 5);
            \path[black!50] (A 5) edge (A 1);
            \path[black,line width=2] (A 6) edge (A 3);
            \path[black,line width=2] (A 6) edge (A 4);
            \path[black!50] (A 1) edge[bend left] (A 4);
        \end{tikzpicture}
        \caption{$B_\lambda^*-g+r$ \\ \hspace{\textwidth}}
        \label{fig: B-g+r after 4}
    \end{subfigure}
    \hfill
    \begin{subfigure}[b]{0.23\textwidth}
        \centering
            \begin{tikzpicture}[scale=0.5,every node/.style={scale=1,circle}]
            \graph  [nodes={draw,circle,fill=black!25,empty nodes}, edges={black!50}, counterclockwise, radius=8cm] 
            {subgraph C_n [n=5,m=3,radius=1.3cm,name=A]};
            \node at  ($(A 3)!.5!(A 4)$) (B){};
            \node[draw,circle,fill=black!25] at ($(B)!.5!(A 1)$) (A 6){};
            \foreach \i in {3,4}{
                \draw (A 6) -- (A \i);
            }
            \path[black,line width=2] (A 1) edge (A 2);
            \path[black!50,thin] (A 3) edge (A 4);
            \path[black,line width=2] (A 4) edge (A 5);
            \path[white,line width=2] (A 5) edge (A 1);
            \path[dashed,black,line width=2] (A 5) edge (A 1);
            \path[black,line width=2] (A 6) edge (A 3);
            \path[white,line width=2] (A 6) edge (A 4);
            \path[black,line width=2] (A 6) edge (A 4);
            \path[black!50] (A 1) edge[bend left] (A 4);
            \draw[white,line width=2] (A 2) -- (A 3);
        \end{tikzpicture}
        \caption{$B_\lambda^*-g+r-r+e$ \\ \hspace{\textwidth}}
        \label{fig: B-g+r-r+e after 4}
    \end{subfigure}
    \hfill
    \begin{subfigure}[b]{0.23\textwidth}
        \centering
            \begin{tikzpicture}[scale=0.5,every node/.style={scale=1,circle}]
            \graph [nodes={draw,circle,fill=black!25,empty nodes}, edges={black!50}, counterclockwise, radius=8cm] 
            {subgraph C_n [n=5,m=3,radius=1.3cm,name=A]};
            \node at  ($(A 3)!.5!(A 4)$) (B){};
            \node[draw,circle,fill=black!25] at ($(B)!.5!(A 1)$) (A 6){};
            \foreach \i in {3,4}{
                \draw (A 6) -- (A \i);
            }
            \path[black,line width=2] (A 1) edge (A 2);
            \path[white,line width=2] (A 2) edge (A 3);
            \path[black!50] (A 3) edge (A 4);
            \path[black,line width=2] (A 4) edge (A 5);
            \path[white,line width=2] (A 5) edge (A 1);
            \path[black,line width=2] (A 6) edge (A 3);
            \path[white,line width=2] (A 6) edge (A 4);
            \path[black,line width=2] (A 6) edge (A 4);
            \path[black,line width=2] (A 1) edge[bend left] (A 4);
        \end{tikzpicture}
        \caption{$B_\lambda^*-g+r-r+e-e+q$}
        \label{fig: B-g+r-r+e-e+q after 4}
    \end{subfigure}
    \caption{After the breakpoint $\lambda=4$, the set $F_1$ is updated to $F_2=\Set{g,r,e}$ to maintain the structure from \cref{cor: at most k+l choose l subsets}.
    The corresponding interdicted minimum spanning tree $B_\lambda^{F_2}$ with weight $w(B_\lambda^{F_2},\lambda)=12+\lambda$ is given by the edges $a,b,c,f,q$ and, therefore, differs in the edges $f$ and $q$ from $B_\lambda^{F_1}$, cf.\ \cref{fig: B-g+r-r+e-e+q after 4}.
    In particular, the replacement edge of $e$ with respect to $B_\lambda^{\Set{g,r}}$ before $\lambda=4$ differs from that of $f$ with respect to $B_\lambda^{\Set{g,r}}$ after $\lambda=4$ and must be re-computed.}
    \label{fig: Fi after 4}
\end{figure}

\section{Conclusion}
In this article, we have investigated the parametric matroid $\ell$-interdiction problem for a fixed natural number $\ell$.
Given an arbitrary matroid whose elements are associated with a weight that depends linearly on a real parameter from a given interval, the goal is to compute, for each parameter value, a set of $\ell$-most vital elements and the corresponding objective value.
Such a set consists of $\ell$ elements whose removal increases the weight of a minimum weight basis as much as possible.
We have shown that there are at most $\mathcal{O}(\m^{\ell+1} k^{\frac{1}{3}}\alpha(k\ell))$ and $\O{m^2(k+\ell)^{\ell-1}k}$ many changepoints at which the optimal solution to the problem changes.
Further, we have developed three exact algorithms the running times of which depend on different matroid operations.
All algorithms have polynomial running time if a single independence test can be performed in polynomial time.

\section*{Acknowledgments}
This work was partially supported by the project "Ageing Smart -- Räume intelligent gestalten" funded by the Carl Zeiss Foundation and the DFG grant RU 1524/8-1, project number 508981269. The authors would like to thank Oliver Bachtler for helpful discussions and the resulting improvement of the article.
	
\bibliography{bib_matroid}
	
\end{document}